\documentclass[pdflatex,sn-mathphys-num]{sn-jnl}% Math and Physical Sciences Numbered Reference Style
 \usepackage{mathrsfs}
\usepackage{graphicx}%
\usepackage{multirow}%
\usepackage{amsmath,amssymb,amsfonts}%
\usepackage{amsthm}%
\usepackage{mathrsfs}%
\usepackage[title]{appendix}%
\usepackage{xcolor}%
\usepackage{textcomp}%
\usepackage{manyfoot}%
\usepackage{booktabs}%
\usepackage{algorithm}%
\usepackage{algorithmicx}%
\usepackage{algpseudocode}%
\usepackage{listings}%
\theoremstyle{thmstyleone}%
\newtheorem{theorem}{Theorem}%  meant for continuous numbers
\theoremstyle{thmstyletwo}%
\newtheorem{example}{Example}%
\newtheorem{remark}{Remark}%

\theoremstyle{thmstylethree}%
\newtheorem{lemma}{Lemma}%
\begin{document}

\title[Article Title]{\Large  Local Multilevel Preconditioned Jacobi-Davidson Method for Elliptic Eigenvalue Problems on Adaptive Meshes}


\author[1]{\fnm{Jianing} \sur{Guo}}\email{2111167@tongji.edu.cn}
\author*[1]{\fnm{Qigang} \sur{Liang}}\email{qigang$\_$liang@tongji.edu.cn}
\author[1]{\fnm{Xuejun} \sur{Xu}}\email{xuxj@tongji.edu.cn}
\affil[1]{\orgdiv{School of Mathematical Science, Tongji University, Shanghai, 200092, China and Key Laboratory of Intelligent Computing and Applications, Tongji University, Ministry of Education}}  

\abstract{In this work, we propose an efficient adaptive multilevel preconditioned Jacobi-Davidson (PJD) method for eigenvalue problems with singularity. Our multilevel method utilizes a local smoothing strategy to solve the preconditioned Jacobi-Davidson algebraic systems arising from adaptive finite element methods (AFEM). As a result, the algorithm holds optimal computational complexity $O(N)$. The theoretical analysis reveals that our method has a uniform convergence rate with respect to mesh levels and degrees of freedom. Further, the convergence rate is not affected by highly discontinuous coefficients within the domain. Numerical results verify our theoretical findings.}

\keywords{PDE eigenvalue problems, adaptive finite elements, local multilevel method, preconditioned Jacobi-Davidson method}



\maketitle

\section{Introduction}\label{sec1}

\par Developing efficient numerical methods for solving large-scale eigenvalue problems plays an important role in scientific and engineering computations. This paper focuses on eigenvalue problems with singularity, which are commonly encountered in practical applications, like non-convex domains or highly discontinuous coefficients.
These singularities present a significant challenge in numerical solutions, as they degrade the accuracy of finite element methods (FEM), especially when a quasi-uniform triangulation is used. Uniformly refining the entire grid often leads to exponential growth with respect to the scale of the discrete system, making it impractical for large-scale eigenvalue problems. To overcome this difficulty, adaptive finite element methods (AFEM) have been rapidly developed, since they may effectively capture the singular behavior of eigenfunctions. The AFEM procedure involves four steps: 
\begin{equation}\notag
\text { SOLVE } \rightarrow \text { ESTIMATE } \rightarrow \text { MARK } \rightarrow \text { REFINE. }
\end{equation}
Although the ESTIMATE step has been extensively studied (\cite{dai2008convergence, czm2002, MR3407250, MR3347459, MR4136540}), the design of algorithms and the corresponding convergence analysis for the SOLVE step in AFEM for solving eigenvalue problems remain limited. Notably, Carstensen and Gedicke \cite{MR2970733}  demonstrate that the adaptive finite element method with a suitable iterative algebraic eigenvalue solver holds asymptotic quasi-optimal computational complexity. However, it is still a great challenge to design an efficient eigensolver for solving eigenfunction on adaptive meshes, particularly due to the incompatibility of traditional preconditioners liking overlapping domain decomposition preconditioners \cite{MR2104179} with locally refined grids. In this paper, we propose an efficient adaptive local multilevel preconditioned Jacobi-Davidson (PJD) method to solve the discrete elliptic eigenvalue algebraic systems generated by the AFEM.
\par The Jacobi-Davidson (JD) method has gained popularity in the last decade. However, its correction operator is often ill-conditioned when dealing with large-scale discrete eigenvalue problems. Thus, the design of effective preconditioning techniques becomes crucial. Zhao et al.\,\cite{MR3499456} propose a parallel two-level domain decomposition based Jacobi–Davidson algorithms for pyramidal quantum dot simulation. Wang and Xu \cite{wang2019convergence} give a rigorous analysis for a two-level preconditioned Jacobi-Davidson method based on the overlapping domain decomposition for computing the principal eigenpairs of $2m$-th order ($m=1,2$) elliptic eigenvalue problems. Recently, a two-level block preconditioned Jacobi-Davidson method has been proposed in \cite{liang2022two} to compute the first several eigenpairs, including multiple and clustered cases. Although these methods significantly reduce computational costs, we should note that domain decomposition methods typically require quasi-uniform grids, which may limit their applicability in computing discrete systems resulting from the AFEM approximations for eigenvalue problems with singularity. Our proposed method overcomes this limitation by introducing a novel adaptive local multilevel preconditioning strategy, specifically designed for eigenvalue algebraic systems arising from AFEM.

\par Multilevel methods are well-suited as preconditioners for algebraic systems generated by the AFEM and have been extensively studied in the literature (\cite{bramble2019multigrid, chen2015multigrid} and references therein). Mitchell \cite{mitchell1992optimal} points out that traditional multilevel algorithms always perform smoothing at all nodes, which may result in the worst computational complexity $O(N^2)$. Numerical results in \cite{mitchell1992optimal} demonstrate that local smoothing may achieve optimal computational complexity, i.e., $O(N)$. Building upon this, Wu and Chen \cite{wu2006uniform} propose an adaptive local multilevel algorithm, which only performs Gauss-Seidel smoothing on new nodes and old nodes whose corresponding basis function support has been changed. Subsequently, Xu et al.\,\cite{xu2010optimality} propose a detailed strategy for determining the locations of local smoothing and demonstrate the convergence of the local multilevel algorithm. Specifically, their proposed preconditioners are designed for linear source problems but not suitable for eigenvalue problems.
\par Motivated by the background above, we propose an adaptive multilevel PJD method to deal with eigenvalue problems with singularity. This method combines local successive subspace correction with shifted techniques to design a local multilevel preconditioner for solving the Jacobi-Davidson correction equation in each outer iteration. As a result, we only need to compute a small-scale eigenvalue problem, significantly reducing the computational cost compared to the original eigenvalue problem. Through a rigorous convergent analysis, we prove that our method has a uniform convergence rate independent of mesh levels and degrees of freedom. Further, the convergence rate remains stable for highly discontinuous coefficients. Numerical results demonstrate that our method accurately captures the locations of singularities while achieving the optimal computational complexity $O(N)$. 
\par The outline of this paper goes as follows. In Section 2, we introduce the model problem and some preliminaries. Section 3 presents the proposed adaptive multilevel PJD method. In Section 4, we give some lemmas used for the error estimates of the proposed method. The uniform convergence rate of the iterative method is established in Section 5. Section 6 presents numerical results which support our theoretical findings.

\section{Model problem}\label{sec2}

\par In this section, we introduce the model problem and some basic notations. We adopt the standard notations for Sobolev spaces $H^s(\Omega)$ with their corresponding norms $\| \cdot\|_{s,\Omega}$ and seminorms $| \cdot|_{s,\Omega}$. In particular, we define $H_0^1(\Omega)$ as a subspace of $H^1(\Omega)$ with a vanishing trace.
\par Consider the elliptic eigenvalue problem
\begin{equation}\label{modelproblem}
   \begin{cases}
      -\nabla \cdot (\rho(x)\nabla u) =\lambda u\ \ \ \ &\text{in} \ \ \Omega,\\
     \ \ \ \ \ \ \ \ \ \ \ \ \ \ \ \ \,\ u=0\ \ \ \  &\text{on} \ \ \partial \Omega,
   \end{cases}
\end{equation}
where $\Omega $ is a polygonal domain in $\mathbb{R}^2$. The coefficient $\rho(x)\left(\rho(x)\geq 1 \right)$ is a piecewise constant function that may exhibit large jumps in $\Omega$. In this paper, we are interested in the first eigenvalue and corresponding eigenfunction. The variational form of \eqref{modelproblem} is to find $(\lambda, u) \in \mathbb{R} \times V$ such that
\begin{equation}\label{variational}
a_{\rho} (u,v) = \lambda b(u,v) \ \ \ \  \forall\ v \in V,
\end{equation}
where $V:= H_0^1(\Omega)$, the bilinear forms $a_{\rho}(\cdot,\cdot): V \times V \to \mathbb{R}$ and $b(\cdot,\cdot): L^2(\Omega) \times L^2(\Omega) \to \mathbb{R}$ are defined as follows:
\begin{align*}
a_{\rho}(u,v) &:= \int_{\Omega} {\rho}(x) \nabla u \cdot \nabla v dx \ \ \ \  \forall\ u, v \in V, \\
b(u,v) &:= \int_{\Omega} u \cdot v  dx \ \ \ \ \quad \quad \quad \ \, \forall\ u, v \in L^2(\Omega).
\end{align*}
Since both bilinear forms are symmetric and positive-definite, we may define the norms $\|v\|_A^2:=a_{\rho}(v,v)$ for all $v\in V$, $\|v\|_0^2:=b(v,v)$ for all $v \in L^2(\Omega)$. 
The eigenvalues of \eqref{variational} are (see \cite{babuvska1989finite})
\begin{equation*}
    0<\lambda_1 \leq \lambda_2 \leq \lambda_3 \leq \cdots \leq \lambda_n \to +\infty,
\end{equation*}
and the corresponding eigenvectors are
\begin{equation*}  
u_1,u_2,u_3,\cdots,u_n,\cdots,
\end{equation*}
which satisfy $a_{\rho}(u_i,u_k)=\lambda_i b(u_i,u_k)=\lambda_{i} \delta_{ik}$ with $\delta_{ik}$ being Kronecker delta.

\par To obtain the finite element discretization, we begin with a quasi-uniform and shape-regular triangulation denoted as $\mathcal{T}_0$, with $h_0$ representing the maximal diameter of its elements. By using the newest vertex bisection technique \cite{MR2050077}, we generate a series of shape-regular triangulations denoted as $\{\mathcal{T}_l\}_{l=1}^L$. Following this, we consider the piecewise linear and continuous finite element space on $\mathcal{T}_l\ (l =0,1, \cdots, L)$, denoted by $V_{0}\subset V_{1}\subset V_{2}\subset ... \subset V_{L}$. It is worth noting that $L$ is not fixed and will increase as the AFEM procedure progresses. 
\par The discrete variational form of \eqref{modelproblem} on $\mathcal{T}_l $ is to find $(\lambda_{i,l}, u_{i,l}) \in \mathbb{R}\times V_l $ such that
\begin{equation}\label{discretevariational}
    a_{\rho}(u_{i,l},v_{l})=\lambda_{i,l} b(u_{i,l},v_{l})\ \ \ \ \forall\ v_{l} \in V_l,
\end{equation}
where $i=1,\cdots,N_l$ and $N_l=$dim$(V_l).$
The eigenvalues of \eqref{discretevariational} are 
$$\lambda_{1,l} \leq \lambda_{2,l}\leq \cdots \leq \lambda_{{N_l},l} $$
and the corresponding eigenvectors are
$u_{1,l},u_{2,l},\cdots,u_{N_l,l}$, which satisfy 
$$a_{\rho}(u_{i,l},u_{j,l})=\lambda_{i,l} b(u_{i,l},u_{j,l})=\lambda_{i,l} \delta_{ij}\ \ \ \ \forall\ i,j=1,\cdots,N_l.$$
Consider the linear operator $A_{l}: V_l \to V_l$ defined as follows:
\begin{equation}\notag
b(A_{l} v_l,w_l)=a_{\rho}(v_l,w_l)\ \ \ \ \forall\ v_l,w_l\in V_{l}.
\end{equation}
It follows that $A_l u_{i,l} = \lambda_{i,l} u_{i,l}$ for $i=1,2,\ldots,N_l$. 
Moreover, for each $V_l$ $(l = 0, 1, \cdots, L)$, the following spectral decomposition holds:
\begin{equation}\label{directsumonL}
V_l= U_1^l \oplus U_2^l,
\end{equation}
where $U_1^l=$span$\{ u_{1,l}\}$, $\oplus$ represents the orthogonal direct sum with respect to $b(\cdot,\cdot)$ (also $a_{\rho}(\cdot,\cdot)$) and $U_2^l$ denotes the orthogonal complement of $U_1^l$ with respect to $b(\cdot,\cdot)$ (also $a_{\rho}(\cdot,\cdot)$).
Further, we denote by $Q_1^l:V_{l}\to U_{1}^{l},\ Q_2^l:V_{l}\to U_{2}^{l}$ the $b(\cdot,\cdot)$-orthogonal projectors. 

\par To make the ideas clearer, we shall introduce some notations commonly used in local multilevel methods.
Let $\mathcal{M}_l$ denote the set of interior nodes on $\mathcal{T}_l$. For any node $\boldsymbol{z}\in \mathcal{M}_l$ , $\phi_l^{\boldsymbol{z}} $ represents the associated nodal finite element basis function of $V_l$, and $\Omega_l^{\boldsymbol{z}}$ represents the support of $\phi_l^{\boldsymbol{z}}$, that is, $\Omega_l^{\boldsymbol{z}}:= \text{supp}\{\phi_l^{\boldsymbol{z}}\}.$ We define
$\widetilde{\mathcal{M}}_l$ as the set of new nodes and old nodes whose corresponding basis function support has been changed ( Fig.\ref{figure0}),
$$
\widetilde{\mathcal{M}}_l=\{ \boldsymbol{z} \in \mathcal{M}_l : \boldsymbol{z} \in \mathcal{M}_l \backslash \mathcal{M}_{l-1}\  \text{or}\ \boldsymbol{z} \in \mathcal{M}_{l-1}\ \text{but}\ \phi_l^{\boldsymbol{z}} \neq \phi_{l-1}^{\boldsymbol{z}}\}.
$$
For brevity, we use $\widetilde{m}_l$ to represent the cardinality of $\widetilde{\mathcal{M}}_l$ and denote $\phi_{l,i}:=\phi_l^{\boldsymbol{x}_{l,i}}$, $\boldsymbol{x}_{l,i}\in \widetilde{\mathcal{M}}_l$, where $i=1,2,...,\widetilde{m}_{l}$.
\begin{figure}
  \centering
\includegraphics[width=0.5\textwidth]
  {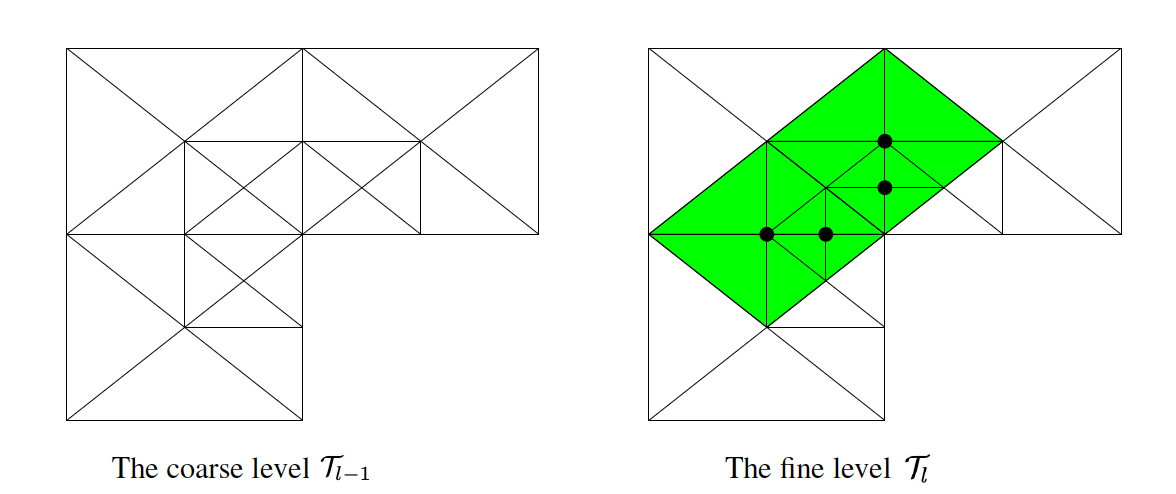}
\caption{$\widetilde{\mathcal{M}}_l\text{ consists of the big dots in the right figure}.$}\label{figure0}
\end{figure}
Then define $Q_{l,i} : V_L \to V_{l,i} :=\text{span}\{ \phi_{l,i}\}$ as a $b(\cdot,\cdot)$-orthogonal projector. We also define a local operator $A_{l,i} : V_{l,i} \to V_{l,i}$ such that
\begin{equation*}
    b(A_{l,i} v_{l,i},w_{l,i})=a_{\rho}(v_{l,i},w_{l,i})\ \ \ \ \forall\ v_{l,i},\ w_{l,i}\in V_{l,i}.    
\end{equation*}
Let $\Omega_{l,i}$ be the support of $\phi_{l,i}$ and $h_{l,i}$ be the diameter of $\Omega_{l,i}$. The minimum eigenvalue of $A_{l,i}$ satisfies 
\begin{equation}\label{localeigenvalue}
    \lambda_{\min}(A_{l,i})\geq O(h_{l,i}^{-2})\geq O(h_{0}^{-2}).
\end{equation}

\section{The adaptive multilevel PJD method}\label{sec3}
\par In this section, we propose an adaptive multilevel PJD method (Algorithm 1) to solve the principal eigenpair of elliptic eigenvalue problems. Our method is specifically designed for the SOLVE step within the AFEM procedure. Assuming the current level in the AFEM procedure is $L$, the iterative solution at $L-1$ has been given by using Algorithm 1, represented as $(\lambda_1^{L-1}, u_1^{ L-1})$. Then, we may apply Algorithm 1 once more to compute the solution $(\lambda_1^{L}, u_1^{L})$ at level $L$. Continuing, we proceed with the remaining steps of the AFEM procedure to move to the next level $L+1$. 
\par For brevity, we define the operators $\tilde{A}_0^{j} :=A_0 -\lambda^{j}I$, $ \tilde{A}_L^{j} :=A_L -\lambda^{j}I$ and $\tilde{A}_{l,i}^{j}:=A_{l,i}-\lambda^{j}I$ $(l=1,2,\cdots,L,\ i=1,2,\cdots,\widetilde{m}_l)$ as follows:
\begin{equation*}
    \begin{array}{lll}
    & b(\tilde{A}_0^{j} v_0,w_0)=(v_0,w_0)_{E^{j}} ,&\forall\ v_0,w_0\in V_0,\\
        & b(\tilde{A}_L^{j} v_L,w_L)=(v_L,w_L)_{E^{j}} ,&\forall\ v_L,w_L\in V_L,\\
        & b(\tilde{A}_{l,i}^{j} v_{l,i},w_{l,i})=(v_{l,i},w_{l,i})_{E^{j}} &\forall\ v_{l,i},w_{l,i} \in V_{l,i},
    \end{array}
\end{equation*}
where $(\cdot,\cdot)_{E^{j}}= a_{\rho}(\cdot,\cdot)-\lambda^{j}b(\cdot,\cdot)$, $\lambda^{j}$ represents the $j$-th iterative approximation of $\lambda_{1,L}$, and the notation $I$ is the identity operator. Note that while $(\cdot, \cdot)_{E^{j}}$ is a bilinear form, it may not always be an inner product. {For a given $\rho(x)$, based on \eqref{localeigenvalue}, we may choose  sufficiently small $h_0$ to ensure that $\tilde{A}_0^{j}$ and $\tilde{A}_{l,i}^{j}$ are well-defined.} We also denote by $Q_l: V_L \to V_l$ the $b(\cdot,\cdot)$-orthogonal projector ($0 \leq l \leq L-1$). Additionally, for any subspace $U \subset V_L$, we define $ U^{\perp}:=\{v: b(v,w)=0,\ \forall\ v \in V_L,\ w \in U\},$
and denote by $Q_U:V_L\to U$ the $b(\cdot,\cdot)$-orthogonal projector. The key idea of our method is to design a local multilevel preconditioner $ B_L^{j}$ (see Definition 3.1) to solve the Jacobi-Davidson correction equation \cite{MR1778354}:
\begin{equation}\label{correctionequation}
    \left\{
    \begin{aligned}
       & \text{Find}\  t^{j+1} \in (U^{j})^{\perp} \ \text{such that }\\
       &  b((A_L-\lambda^{j}I)t^{j+1},v)=b(r^{j},v)\ \ \ \ \forall\ v \in (U^j)^{\perp},
    \end{aligned}
    \right.
\end{equation}
where $U^j=$span$\{ u^{j}\},\ u^{j}$ is the $j$-th iterative approximation of $u_{1,L}$ and $r^{j}=\lambda^{j} u^{j} - A_L u^{j}.$ 
\begin{table}[h]
\centering
\begin{tabular}{p{15cm}}
\hline
\hline
\textbf{Algorithm 1} {The JD method with local multilevel preconditioner} \\
\hline
$\bf{{Step\ 1}}$ Input: $\lambda^{0}=\lambda_1^{L-1},\ u^{0}=u_1^{L-1}$. Set $W^0=$span$\{u^{0}\}.$\\
$\bf{{Step\ 2}}$ For $j=0, 1, 2, ...,$ solve \eqref{correctionequation} inexactly through solving a preconditioned system:
\begin{equation}\label{jdequation}
{t}_L^{j+1}=Q_{\bot}^j B_{L}^{j} r^{j},
\end{equation}
\ \ \ \ \ \ \ \ \ \ \ where $B_{L}^{j}$ shall be defined in Definition 3.1 and $Q_{\bot}^{j}:=(I-Q_{U^j})$.\\
$\bf{{Step\ 3}}$ Solve the first eigenpair in  $W^{j+1}$:\\
\begin{equation}\notag
a_{\rho}(u^{j+1},v)=\lambda^{j+1}b(u^{j+1},v)\ \ \ \ \forall\ v \in W^{j+1},\ \|u^{j+1}\|_0=1,
\end{equation}
\ \ \ \ \ \ \ \ \ \ \ where $W^{j+1}=W^{j}+$\text{span}$\{{t}^{j+1}_L\}$. Set $U^{j+1}=$span$\{u^{j+1}\}.$\\
$\bf{{Step\ 4}}$ If $|\lambda^{j+1}-\lambda^{j}|<tol$, $(\lambda_{1}^{ L},u_{1}^{L})\leftarrow (\lambda^{j+1},u^{j+1})$. Otherwise, goto $\bf{{Step\ 2}}$.\\
\hline
\hline
\end{tabular}
\end{table}
\par Next, we introduce the definition of the local multilevel preconditioner $B_L^{j}$.
\begin{table}[h]
\centering

\begin{tabular}{p{15cm}}
\hline
\hline
\textbf{Definition 3.1.} Local multilevel preconditioner $B_L^{j}$\\
\hline
For $L>0$, define $B_{L}^{j} r^{j}$ as follows:\\
\textbf{Step 1} Input: $\tilde{t}_0=0$.\\
\textbf{Step 2} For $l=0,1,\cdots,L-2,L-1$ do:\\
\qquad $\tilde{t}_{l+1}=\tilde{t}_{l}+ R_l^{j} Q_l \bigl(r^{j}-\tilde{A}_L^{j}\tilde{t}_{l}\bigr)$\\
\qquad End.\\
\textbf{Step 3} Output: $B_L^{j} r^{j}=\tilde{t}_{L}$.\\
\hline
\hline
\end{tabular}
\end{table}
Within the algorithm, $R_l^{j}: V_L \to V_{l,i}$ is the local Jacobi smoother given by
\begin{equation}\label{jacobismoother}
   R_l^{j} := \left\{ 
   \begin{aligned}
      & (\tilde{A}_0^{j})^{-1}Q_2^0 Q_0\ \ \ \ \ \ \ \  l=0, \\
      &\gamma \sum_{i=1}^{\widetilde{m}_l} (\tilde{A}_{l,i}^{j})^{-1} Q_{l,i}\ \ 
    \ \ 1\leq l \leq L,
   \end{aligned} \right.
\end{equation}
where $\gamma$ is an appropriately chosen positive scaling factor ($0<\gamma <1$). It is easy to check that $R_l^{j}$ is symmetric in the sense of $L^2$-inner product for $l=0,1,\cdots, L$. 

\section{Preliminaries on multilevel subspaces}
In this section, we introduce some properties of multilevel subspace, which shall be used in the following convergence analysis. 
% For simplicity, we focus on the case $a(x)\equiv 1$ in the rest of this paper. 
Throughout this paper, the notations $C$ (with or without subscripts) represent generic positive constants independent of mesh levels, degrees of freedom, and discontinuous coefficients. Note that these constants may be different at different occurrences. To help the readers understand the convergence theory, we emphasize the meanings of the indices. We denote by $l\ (l=0,1,2,...,L)$ the level of adaptive meshes, $j\ (j=1,2,...)$ the current number of iteration in Algorithm 1, $i\ (i=1,2,...,\widetilde{m}_{l})$ the index of the location for the local smoothing on the $l$-th level mesh.
\par First, based on Algorithm 1, we have
$$\text{span}\{ u^{j} \}+\text{span}\{{t}_L^{j+1}\}\subset W^j + \text{span}\{{t}_L^{j+1}\} = W^{j+1} \subset V_L.$$ 
We may set $\lambda^{0}=\tilde{\lambda}_1 < \lambda_{1,0}$ by solving the eigenvalue problem on $\widetilde{\mathcal{T}}_0$ refined from the coarsest mesh $\mathcal{T}_0$. Consequently, we get
\begin{equation}\label{lambdaueq2}
    \lambda_1 \leq \lambda_{1,L} \leq \lambda^{j}\leq \lambda^{0} < \lambda_{1,0} < \lambda_{2,L} \leq \lambda_{2,0},\ \ \ j=0,1,\cdots.
\end{equation}   
   This clearly demonstrates that $(\cdot,\cdot)_{E^{j}}$ is an inner product in $U_2^L$ and $V_0$, inducing the corresponding norm $\|\cdot\|_{E^{j}}$ in both $U_2^L$ and $V_0$. According to \eqref{localeigenvalue}, we may choose sufficiently small $h_0$ to ensure that $(\cdot,\cdot)_{E^{j}}$ is an inner product in $V_{l,i}$, with its induced norm $\|\cdot\|_{E^{j}}$ being equivalent to $\|\cdot\|_A$ in $V_{l,i}$. 
   It is easy to check that the following inequality holds:
\begin{equation}\label{atoej}
  (v,v)_{E^{j}} \leq a_{\rho}(v,v)=(v,v)_{E^{j}}+\lambda^{j}b(v,v)\leq \beta(\lambda_{2,L})(v,v)_{E^{j}}\ \ \ \ \ \forall\ v \in U_2^L,
\end{equation}
where $\beta(\lambda)=1+\frac{\lambda^{j}}{\lambda-\lambda^{j}}.$ Hence, $\|\cdot\|_{E^{j}}$ is equivalent to $\|\cdot\|_A$ in $U_2^L$. Similarly, we may deduce that $\|\cdot\|_{E^{j}}$ and $\|\cdot\|_A$ are also equivalent in $V_0$ and $V_{l,i}\ (l=1,2,\cdots,L,\ i=1,2,\cdots,\widetilde{m}_l)$.
\par To give the a priori results for eigenvalue problems, we denote by $M(\lambda_1)$ the eigenspace corresponding to the eigenvalue $\lambda_{1}$. Define
 \begin{equation*}
     \begin{aligned}
         \delta_l \left(\lambda_1\right)&:=\sup _{w \in M\left(\lambda_1\right),\|w\|_0=1} \inf _{v_l \in V_l}\left\|w-v_l\right\|_{A},\\
          \eta_A \left(V_l\right)&:=\sup _{f \in L^2(\Omega),\|f\|_{0}=1} \inf _{v_l \in V_l}\left\|\psi-v_l\right\|_{A},
     \end{aligned}
 \end{equation*}
 where $\psi$ satisfy
 $a_{\rho}(\psi,v)=b(f,v)$ for all $v\in V$ and $l=0,1,\cdots,L$.
%There are some usual priori error estimates which are stated as follows in \cite{babuvska1989finite}.
\begin{lemma}\label{ukulemma}  (see \cite{babuvska1989finite,MR2652780}).
    For ($\lambda_{1},u_{1}$), there exists a discrete eigenpair $(\lambda_{1,0},u_{1,0})$ on $V_{0}$ such that
\begin{equation}\notag
%\label{uku}
    \begin{aligned}
        \left\|u_1-u_{1,0}\right\|_{A} & \leq  C\delta_{0}(\lambda_1), \\
        \left\|u_1-u_{1,0}\right\|_{0} & \leq C \eta_A \left(V_0\right)\left\|u_1-u_{1,0}\right\|_{A}, \\
        \left|\lambda_1-\lambda_{1,0}\right| & \leq C\left\|u_1-u_{1,0}\right\|_{A}^2,
    \end{aligned}
\end{equation}
where the constant $C$ depends on the gap of $\lambda_1$.
\end{lemma}
Combining this lemma with \eqref{lambdaueq2} yields
\begin{equation}\label{lambdajto1l}
    \lambda^{j} -\lambda_{1,L} < \lambda_{1,0}-\lambda_1 \leq C \delta_0^2(\lambda_1),
\end{equation}
where $ \delta_0(\lambda_1) \to 0$ as $h_0 \to 0.$ Then similar to Lemma 4.3 in \cite{wang2019convergence}, we have
\begin{lemma}\label{q12}
 { It holds that }
    \begin{align*}
            \ \|Q_1^L Q_2^0 v_0\|_0 &\leq C \delta_{0}(\lambda_1)\eta_A \left(V_0\right)\|Q_2^0 v_0\|_0 \ \ \ \ \forall\ v_0\in V_0      
    \end{align*}
    and
            \begin{equation}\notag 
            \|Q_2^0 Q_0 Q_1^L v_L \|_0 \leq C \delta_{0}(\lambda_1) \eta_A \left(V_0\right) \| Q_1^L v_L\|_0\ \ \ \ \forall\ v_L\in V_L.
      \end{equation}
 \end{lemma}
\par Clearly, the terms $\|Q_2^L Q_1^0 v_0\|_0$ and $\|Q_1^0 Q_0 Q_2^L v_L\|_0$ have similar estimates. Next, we state two important properties of the abstract Schwarz theory: The stability of decomposition for the error subspace $U_{2}^{L}$ and the strengthened Cauchy-Schwarz inequality for local subspaces $V_{l, i},\ l=0,1,..., L, i=1,2,...,\widetilde{m}_{l}$.
It should be noted that Lemmas \ref{Stability}, \ref{strength} hold for sufficiently small $h_0$.
\begin{lemma}[The stability of decomposition]\label{Stability}
For any $ v \in U_2^L$, there exists $v_0 \in V_0 , v_{l,i}\in V_{l,i}$ such that
$$ v=Q_2^L v_0 + Q_2^L \sum_{l=1}^L \sum_{i=1}^{\widetilde{m}_l} v_{l,i},$$
and
\begin{equation}\label{stableequation}
    (v_0,v_0)_{E^{j}} + \sum_{l=1}^L \sum_{i=1}^{\widetilde{m}_l}(v_{l,i},v_{l,i})_{E^{j}} \leq C C_h^{\rho} (v,v)_{E^{j}},
\end{equation}
where $ C_h^{\rho}= \min \{ {|{\rm log} h_{\min}|^2, \rho_{\max}}\}\left(\rho_{\max}=\max_{x \in \Omega} \rho(x)\right)$.
\end{lemma}
\begin{proof}
    Let $v_0 = \pi_0 v,$ $v_{l,i}=v_l(\boldsymbol{x}_{l,i})\phi_{l,i},$
    where $v_l = (\pi_l -\pi_{l-1})v$ and $\pi_l\,(l=0,1,\cdots,L)$ is the local quasi-interpolation operator defined in \cite{MR2928972}. It is easy to see that the first equality holds. 
    \par By the property of $\pi_{0}$, we have 
    \begin{equation}\label{pi0v}
        (v_0,v_0)_{E^{j}}\leq a_\rho (v_0,v_0) =\|v_0\|_A^2 \leq C \tilde{C}_h^{\rho} \|v\|_A^2,
    \end{equation}
    where $\tilde{C}_h^{\rho}=\min \{| \text{log} h_{min}|, \rho_{\max}\} $ and $h_{\min}$ is the minimum diameter of the triangles on the finest mesh. Further, it is proved in \cite{MR2928972} that $\sum_{l=1}^L \sum_{i=1}^{\widetilde{m}_l}\|v_{l,i}\|_A^2 \leq C C_h^{\rho} \|v\|_A^2.$ Combining this with \eqref{atoej} and \eqref{pi0v}, we get
     $$(v_0,v_0)_{E^{j}}+\sum_{l=1}^L \sum_{i=1}^{\widetilde{m}_l}(v_{l,i},v_{l,i})_{E^{j}}\leq \sum_{l=1}^L \sum_{i=1}^{\widetilde{m}_l}\|v_{l,i}\|_A^2 +C\|v\|_A^2 \leq C \|v\|_A^2 \leq C \beta(\lambda_{2,L}) C_h^{\rho} \|v\|_{E^{j}}^2,$$ which completes the proof.  
\end{proof}
%\par Next we present another important lemma for our convergence analysis.
\begin{lemma}[Strengthened Cauchy-Schwarz inequality]\label{strength} For any $v_{l,i}$, $w_{l,i} \in V_{l,i}$, it holds that
    $$
\sum_{l=0}^L \sum_{i=1}^{\widetilde{m}_l} \sum_{k=0}^{l-1} \sum_{s=1}^{\widetilde{m}_k}  \left(v_{l,i}, w_{k,s}\right)_{E^{j}} \leq C  \left(\sum_{l=0}^L \sum_{i=1}^{\widetilde{m}_l}\left\|v_{l,i}\right\|_{E^{j}}^2\right)^{\frac{1}{2}}\left(\sum_{l=0}^L \sum_{i=1}^{\widetilde{m}_l}\left\|w_{l,i}\right\|_{E^{j}}^2\right)^{\frac{1}{2}}.
$$
\end{lemma}
\begin{proof}
By \eqref{localeigenvalue}, we may choose sufficiently small $h_0$ to ensure that $(\cdot,\cdot)_{E^{j}}$ is an inner product in each element.  Then the proof follows a similar technique as Theorem 5.1 in \cite{MR2928972}.
\end{proof}
\par Define the operator $K_{l,i}^{j}:V_L \to V_{l,i}$ as follows:
\begin{equation}\label{kli}
    (K_{l,i}^{j} v_L ,w_{l,i})_{E^{j}}=(v_L,w_{l,i})_{E^{j}}\ \ \ \ \forall\ v_L \in V_L , \ w_{l,i} \in V_{l,i},
\end{equation}
where $l=1,2,\cdots,L ,i=1,2,\cdots,\widetilde{m}_l.$
 Based on the Lax-Milgram theorem and \eqref{localeigenvalue}, the operator $K_{l,i}^{j}$ is well-defined for sufficiently small $h_0$, and
 \begin{equation}\label{qaak}
   Q_{l,i} \tilde{A}_{L}^{j} = \tilde{A}_{l,i}^{j} K_{l,i}^{j},
 \end{equation}
which, together with \eqref{jacobismoother}, yields
\begin{equation}\label{kljmini}
 K_l^{j}=\gamma \sum_{{i}=1}^{\widetilde{m}_l} K_{l,i}^{j}, 
\end{equation}
where $K_l^{j}=R_l^{j} Q_l \tilde{A}_L^{j}$. Since $R_l^{j}$ is symmetric with respect to $b(\cdot, \cdot)$, we get
\begin{equation}\label{kljsymmetry}
    (K_l^{j} v_L,w_L)_{E^{j}}=(v_L,K_l^{j} w_L)_{E^{j}}\ \ \ \ \forall\ v_L,w_L \in V_L.
\end{equation}
On the coarsest level, define $K_0^{j}:=R_0^{j} Q_0 \tilde{A}_L^{j},$ then we have
   \begin{equation}\label{k0j}
       (K_0^{j} v_L,\omega_{2}^0)_{E^{j}}=(v_L,\omega_{2}^0)_{E^{j}}\ \ \ \ \forall\ v_L \in V_L,\ \omega_{2}^0 \in U_2^0,
   \end{equation}
which yields that \eqref{kljsymmetry} also holds for $l=0.$
\par Based on Lemma \ref{Stability}, Lemma \ref{strength} and above properties, we may obtain the following crucial lemma for the theoretical analysis.
\begin{lemma}\label{Kjlowerbound}
    Let $K^{j}=\sum_{l=0}^L K_l^{j}.$ Then for sufficiently small $h_0$,
    \begin{equation}\label{kjlower1}
    \|w\|_{E^{j}}^2 \leq (K^{j} w,w)_{E^{j}} \ \ \ \ \forall\ w \in  U_2^0,
    \end{equation}
    and 
    \begin{equation}\label{kjlower2}
         \|v\|_{E^{j}}^2 \leq \frac{C C_h^ {\rho} }{\left(1-C\tilde{C}_h^{\rho}\delta_{0}(\lambda_1)  \eta_A (V_0)\right)^2}(K^{j} v,v)_{E^{j}} \ \ \ \ \forall\ v\in U_2^L.
    \end{equation}
\end{lemma}
\begin{proof}
  First, for all $w \in U_2^0$, the property of $\pi_0$ implies that $\pi_0 w=w$.
   In view of \eqref{kli} and \eqref{k0j}, we get \eqref{kjlower1}. 
Moreover, for all $ v \in U_2^L $, it follows from Lemma \ref{Stability} that
    \begin{equation}\label{vtovl}
        (v,v)_{E^{j}}=\sum_{l=0}^{L}(v_l,v)_{E^{j}},
    \end{equation}
    where $v_0=\pi_0 v$ and $v_l=\sum_{i=1}^{\widetilde{m}_l} v_{l,i}\ (l>0).$
    For the case $l>0$, owing to \eqref{kli}, we have
    \begin{equation*}
     \begin{aligned}
(v_l,v)_{E^{j}}&=\sum_{i=1}^{\widetilde{m}_l}(v_{l,i},v)_{E^{j}}=\sum_{i=1}^{\widetilde{m}_l}(v_{l,i},K_{l,i}^{j} v)_{E^{j}}\leq \sum_{i=1}^{\widetilde{m}_l}(v_{l,i},v_{l,i})_{E^{j}}^{1/2}(K_{l,i}^{j} v,K_{l,i}^{j} v)_{E^{j}}^{1/2}\\
       & \leq \left(\sum_{i=1}^{\widetilde{m}_l}(v_{l,i},v_{l,i})_{E^{j}}\right)^{1/2}\left(\sum_{i=1}^{\widetilde{m}_l}(K_{l,i}^{j} v,K_{l,i}^{j} v)_{E^{j}}\right)^{1/2}.
     \end{aligned}\end{equation*}
For the case $l=0$, from \eqref{directsumonL}, we get
     $(v_0,v)_{E^{j}}=(Q_1^0 v_0,v)_{E^{j}} +(Q_2^0 v_0,v)_{E^{j}}$.
Then using Lemma \ref{q12}, \eqref{atoej} and \eqref{pi0v}, we have
   \begin{equation*}
     \begin{aligned}
       (Q_1^0 v_0,v)_{E^{j}}&=a_\rho( Q_1^0 v_0,Q_1^0 Q_0 v)-\lambda^{j} b(Q_2^L Q_1^0 v_0, v)\\
  %     &\leq |a(Q_2^L Q_1^0 v_0,v)|+\lambda^j| b(Q_1^0 v_0,Q_1^0 Q_0 v)|\\
       &\leq \| Q_1^0 v_0\|_A \|Q_1^0 Q_0 v\|_A+\lambda^{j} \|Q_2^L Q_1^0 v_0\|_0 \| v\|_0\\
       & \leq C \delta_{0}(\lambda_1) \eta_A (V_0) \|Q_1^0 v_0\|_A \|v\|_A+C\lambda^{j} \delta_{0}(\lambda_1)\eta_A\left(V_0\right) \|Q_1^0 v_0\|_0\|v\|_0\\
       & \leq C \tilde{C}_h^{\rho}\delta_{0}(\lambda_1)  \eta_A (V_0)\|v\|_{E^{j}}^2.
     \end{aligned}
   \end{equation*}
Due to \eqref{lambdaueq2}, \eqref{k0j} and the fact that $(Q_1^0 v_0 ,v_0)_{E^{j}}>0$, the following holds:
   \begin{equation}\label{v0v2}
     (Q_2^0 v_0,v)_{E^{j}}=(Q_2^0 v_0,K_0^{j} v)_{E^{j}}\leq( Q_2^0 v_0,v_0)_{E^{j}}^{1/2}(K_0^{j} v,K_0^{j} v)_{E^{j}}^{1/2}\leq (v_0,v_0)_{E^{j}}^{1/2}(K_0^{j} v,v)_{E^{j}}^{1/2}.
   \end{equation}
Finally, combining \eqref{stableequation} and \eqref{vtovl}$\sim$\eqref{v0v2}, we deduce
   \begin{equation*}%\label{eq30}
       \sum_{l=0}^{L} (v_l,v)_{E^{j}}\leq C\tilde{C}_h^{\rho}\delta_{0}(\lambda_1)\eta_A (V_0) \|v\|_{E^{j}}^2 +C\|v\|_{E^{j}}(K^{j} v,v)_{E^{j}}^{1/2},
   \end{equation*}
  which completes the proof of \eqref{kjlower2}.
\end{proof}
\par In order to analyze the error operator in the next section, we first define $\tilde{t}$ as the solution to the preconditioned equation:
\begin{equation}\notag
  B_L^{j}(A_L-\lambda^{j})\tilde{t}=B_L^{j} r^{j}.
 \end{equation}
Subsequently, from Definition 3.1, we get
 $$  \tilde{t}-\tilde{t}_L=\tilde{t}-B_L^{j} r^{j}=(I-B_L^{j} \tilde{A}_L^{j})\tilde{t},$$
 and
 \begin{equation}\label{blmtoelj}
  I-B_L^{j} \tilde{A}_L^{j}=E_L^{j},
 \end{equation}
 where $E_L^{j}=(I-K_L^{j})\cdots (I-K_1^{j})(I-K_0^{j})$ (\cite{bramble2019multigrid}). Define $\tilde{E}_L^{j}:=(I-K_0^{j})(I-K_1^{j})\cdots(I-K_L^{j}),$ then from \eqref{kljsymmetry}, we may deduce that 
 \begin{equation}\label{blmsymmetry}
    (E_L^{j} v_L , w_L)_{E^{j}} = (v_L, \tilde{E}_L^{j} w_L)_{E^{j}},
 \end{equation}
 where $v_L,\ w_L \in V_L$ and $E_{-1}^{j}=I.$ Further, it follows from $E_l^{j}=(I-K_l^{j})E_{l-1}^{j}$ that
 \begin{equation}\label{eljsum}
      I-E^{j}_L=\sum_{l=0}^L K^{j}_l E^{j}_{l-1}.
 \end{equation}
\par To obtain the uniform convergence rate of our algorithm, we need the following classical inequalities, which play a crucial role in analyzing the convergence of multilevel methods.
 \begin{lemma}\label{kjupperbound}
  Let $K^{j}=\sum_{l=0}^L R_l^{j} Q_l \tilde{A}_L^{j}$. For sufficiently small $h_0$, it holds that
  \begin{equation*}
         (K^{j} v_L,v_L)_{E^{j}}\leq C \sum_{l=0}^L (K_l^{j} E_{l-1}^{j} v_L,E_{l-1}^{j} v_L)_{E^{j}}\ \ \ \quad\quad\ \forall\ v_L \in V_L,
  \end{equation*}
and
  \begin{equation*}
         ({2-\omega})\sum_{l=0}^L (K_l^{j} E_{l-1}^{j} v_L,E_{l-1}^{j} v_L)_{E^{j}}\leq (v_L,v_L)_{E^{j}}-(E^{j}_L v_L,E^{j}_L v_L)_{E^{j}} \ \ \ \ \ \ \ \ \,\forall\ v_L \in V_L.
  \end{equation*}
  where $\omega=\max_{l=0,\cdots, L} w_l $ $(0<\omega_l<2)$, depending on the positive scaling factor of Jacobi smoother and the shape regularity of the meshes.
\end{lemma}
\begin{proof}
Based on Lemma \ref{strength}, we may prove this lemma by using similar techniques as Theorem 3.2 in \cite{bramble2019multigrid}. 
\end{proof} 
\par Moreover, since the theoretical analysis is conducted within the error subspace $U_{2}^L$, we need to bound the $L^2$-norm of the error operator. Based on \eqref{kljmini} and Lemma \ref{kjupperbound}, we may get the following estimate. 
\begin{lemma}\label{klfl1}
    It holds that 
    \begin{equation}\notag
        \sum_{l=1}^L b(w_L, K_l^{j} E_{l-1}^{j} v_L) \leq C \|w_L\|_0 \left((v_L,v_L)_{E^{j}}+(\lambda^{j}-\lambda_{1,L}) \|E_L^{j} v_L\|_0^2\right)^{1/2}\ \ \ \ \forall\ v_L,\ w_L\in V_L.
    \end{equation}
\end{lemma}
\par Further, we define $F_l^{j}:=(I-K_l^{j})\cdots(I-K_L^{j})$ and $\tilde{F}_l^{j}:=(I-K_L^{j})\cdots(I-K_l^{j}),\ l=0,1,\cdots,L.$ It is obvious to see that 
\begin{equation}\notag
    F_{l}^{j}=(I-K_{l}^{j})F_{l+1}^{j},
\end{equation}
where $F_{L+1}^{j}=I.$ Then, utilizing the fact that $F_0^{j}=\tilde{E}_L^{j}$, we obtain
\begin{equation*}
   I- \tilde{E}_L^{j}= I-F_0^{j} = \sum_{l=0}^L K_l^{j} F_{l+1}^{j}.
\end{equation*}
In particular, the operator $F_l^{j}$ has similar properties with $E_l^{j}$, as stated in  Lemma \ref{kjupperbound}, Lemma \ref{klfl1}, \eqref{blmsymmetry} and \eqref{eljsum}. We now give a bound for the terms $\|E^{j}_L v_L \|_0$ and $\|\tilde{E}^{j}_L v_L\|_0 $.
\begin{lemma}\label{eljl2upper}
  It holds that 
\begin{equation}\notag
   \|E^{j}_L v_L \|_0 \leq {C}\|v_L\|_{A} \ \ \text{and}\ \ \,\|\tilde{E}^{j}_L v_L\|_0 \leq {C}\|v_L\|_{A}\ \ \ \ \forall\ v_L \in V_L.
\end{equation}
\end{lemma}
\begin{proof}
By the triangle inequality, the Poincar\'e inequality and \eqref{eljsum}, we have
\begin{equation}\label{elb}
    \begin{aligned}
        \|E_L^{j} v_L\|_0 &\leq \|v_L\|_0 + \|(I -E_L^{j} )v_L\|_0\\
        &\leq \sqrt{\frac{1}{\lambda_{1,L}}}\|v_L\|_A+ \|K_0^{j} v_L\|_0 +\|\sum_{l=1}^L K_l^{j} E_{l-1}^{j} v_L\|_0.
    \end{aligned}
\end{equation}
Using \eqref{lambdaueq2} and the definition of $K_0^{j}$, we get 
  \begin{equation*}
 (K_0^{j} v_L,K_0^{j} v_L)_{E^{j}} \geq 0,
  \end{equation*}
  which, together with the Poinca\'re inequality, \eqref{directsumonL}, \eqref{atoej} and \eqref{lambdajto1l}, yields
  \begin{equation}\label{k0jv}
  \begin{aligned}
        \|K_0^{j} v_L\|_{E^{j}}^2&=( K_0^{j} v_L,v_L)_{E^{j}}=( K_0^{j} v_L,Q_1^L v_L)_{E^{j}}+( K_0^{j} v_L,Q_2^L v_L)_{E^{j}}\\
           &= b( K_0^{j} v_L,\tilde{A}_L^{j} Q_1^L v_L)+(Q_2^L K_0^{j} v_L,Q_2^L v_L)_{E^{j}}\\
           &\leq C(\lambda^{j}-\lambda_{1,L})\|K_0^{j} v\|_{0} \| Q_1^L v_L\|_{0}+\|Q_2^L K_0^{j} v_L\|_A \|Q_2^L v_L\|_A\\
           & \leq C(1+\delta_0^2(\lambda_1))\|K_0^{j}v_L\|_{E^{j}}\|v_L\|_A. 
  \end{aligned}
  \end{equation}
Then, taking $w_L=\sum_{l=1}^L K_l^{j} E_{l-1}^{j} v_L$ in Lemma \ref{klfl1}, we obtain
\begin{equation}\label{bel}
      b(\sum_{l=1}^L K_l^{j} E_{l-1}^{j} v_L,\sum_{l=1}^L K_l^{j} E_{l-1}^{j} v_L)
     \leq C\gamma^{1/2} \|\sum_{l=1}^L K_l^{j} E_{l-1}^{j} v_L\|_0\left(\|v_L\|_{A}+(\lambda^{j}-\lambda_{1,L})\|E_L^{j} v_L\|_0\right)^{1/2} .
\end{equation}
Finally, substituting \eqref{k0jv} and \eqref{bel} into \eqref{elb}, we get
\begin{equation}\notag
%\label{bejupperbound}
    \|E^{j}_L v_L \|_0 \leq  C \|v_L\|_{A}.
\end{equation}
The second inequality may be proved by using similar techniques. 
\end{proof}

\section{The main convergent result}
\par Based on the previously established lemmas, in this section we shall present the main convergent result of this paper and its rigorous proof.
\begin{theorem}\label{mainresult}
For sufficiently small $h_0$ such that
   $$\theta=\sqrt{1-\frac{(2-\omega)\left( 1-C \tilde{C}_h^{\rho}\delta_{0}(\lambda_1) \eta_A (V_0)\right)^2}{C C_h^\rho}+C \delta^2_0(\lambda_1)} \in (0,1),$$ it holds that
  \begin{equation}\notag
    \|e_L^{j+1}\|_{E^{j}} \leq  \left\{1- \alpha^{j}(1-\theta)+ C\delta_{0}(\lambda_1)\right\}\|e_L^{j}\|_{E^{j}},
  \end{equation}
  and
  \begin{equation}\notag
    \lambda^{j+1}-\lambda_{1,L} \leq \left\{1- \alpha^{j}(1-\theta)+ C\delta_{0}(\lambda_1)\right\}^2  (\lambda^{j}-\lambda_{1,L}),
  \end{equation}
  where $\omega$ is defined in Lemma \ref{kjupperbound}, $\delta_0(\lambda_1)$ and $\eta_A(V_0)$ are defined in Lemma \ref{ukulemma}, $0<\alpha^{j}<\frac{1}{1-\theta},$ $C_h^{\rho}=\min\{ |{\rm log} (h_{\min})|^2, \rho_{\max}\}$ and $\tilde{C}_h^{\rho}=\min\{ |{\rm log} (h_{\min})|, \rho_{\max}\}.$ When $\rho(x)\equiv 1$, $C_h^{\rho}=\tilde{C}_h^{\rho}=1.$
%  where the $h_0$-dependent constant $c(h_0)$ $\to$ 1 decreasingly when $h_0$ is sufficiently small.
\end{theorem}
\begin{remark}
The coarse mesh size $h_{0}$ is chosen sufficiently small in Theorem \ref{mainresult}
to guarantee that 
    $$\frac{(2-\omega)\left( 1-C \tilde{C}_h^{\rho}\delta_{0}(\lambda_1) \eta_A (V_0)\right)^2}{C C_h^\rho}\in (0,1),$$
which results in $\theta\in (0,1)$. In fact, 
$$\frac{(2-\omega)\left( 1-C \tilde{C}_h^{\rho}\delta_{0}(\lambda_1) \eta_A (V_0)\right)^2}{C C_h^\rho}\to  \frac{(2-\omega)}{C C_h^\rho}   \in (0,1),\ \text{as}\ h_{0}\to 0. $$
    
\end{remark}
\par Before giving the proof of Theorem \ref{mainresult}, we first introduce a special function defined as:
 \begin{equation}\label{ujtilde}
    \tilde{u}^{j+1}:=u^{j}+\alpha^{j} t_L^{j+1} \in U^{j} +\text{span}\{t_L^{j+1}\} \subset W^{j+1},
 \end{equation}
    where $\alpha^{j} $ is an undetermined parameter. Substituting \eqref{jdequation} into \eqref{ujtilde}, we know
    \begin{equation}\label{tildeuj}
\tilde{u}^{j+1}=u^{j}+\alpha^{j} Q_{\bot}^j B_L^{j} r^{j}.
    \end{equation}
Subsequently, we define
\begin{equation}\notag
        e_L^{j}:=-Q_2^L u^{j} \ \ \text{and}\ \ \tilde{e}_L^{j+1}:=-Q_2^L \tilde{u}^{j+1}.
    \end{equation}     
Applying the operator $-Q_2^L$ to both sides of \eqref{tildeuj} yields
    \begin{equation*}
        \tilde{e}_L^{j+1}=e_L^{j}-\alpha^{j} Q_2^L Q_{\perp}^{j}B_L^{j} r^{j}.
      \end{equation*}
 By \eqref{directsumonL} and the fact that $Q_{\perp} ^j = I - Q_{U^{j}}$, we get 
      \begin{equation}
        \label{elj2}
        \begin{aligned}
          \tilde{e}_L^{j+1} &= e_L^{j}-\alpha^{j} Q_2^L B_L^{j} r^{j}+\alpha^{j} Q_2^L Q_{U^{j}} B_L^{j} r^{j}\\
          &= e_L^{j}-\alpha^{j} Q_2^L B_L^{j} (\lambda^{j} I-A_L)(Q_1^L u^{j}-e_L^{j})+\alpha^{j} Q_2^L Q_{U^{j}} B_L^{j} r^{j}\\
          &= \{ e_L^{j}+\alpha^{j} Q_2^L B_L^{j} (\lambda^{j} I-A_L) e_L^{j} \}+\alpha^{j}\{Q_2^L B_L^{j} (A_L-\lambda^{j} I)Q_1^L u^{j}+Q_2^L Q_{U^{j}} B_L^{j} r^{j}\}\\
          &=: R_1^{j}+R_2^{j}.
        \end{aligned}
      \end{equation}
Assume the following estimates of $R_1^{j}$ and $R_2^{j}$ hold:
\par\noindent{\bf (A1)}\ \ $\|R_1^{j}\|_{E^{j}} \leq\left( 1- \alpha^{j} (1- \theta) \right)\|e_L^{j}\|_{E^{j}},$
\par\noindent{\bf (A2)}\ \ $\|R_2^{j}\|_{E^{j}} \leq C\delta_{0}(\lambda_1)\|e_L^{j}\|_{E^{j}},$
\par\noindent where $\theta=\sqrt{1-\frac{(2-\omega)\left(1-C\tilde{C}_h^{\rho}\delta_{0}(\lambda_1)\eta_A(V_0)\right)^2}{C C_h^{\rho}}+C \delta^2_0(\lambda_1)} \in (0,1)$ and $0<\alpha^{j}<\frac{1}{1-\theta}.$ The verifications of (A1) and (A2) shall be presented in the following subsection.

\par\noindent{\bf Proof of Theorem \ref{mainresult}}:\ \
Combining \eqref{elj2}, (A1) and (A2), we deduce
  \begin{equation*}
    \begin{aligned}
    \|\tilde{e}_L^{j+1}\|_{E^{j}} &\leq  \|R_1^{j}\|_{E^{j}} + \|R_2^{j}\|_{E^{j}}\\
      & \leq \left(1-\alpha^{j}(1-\theta)\right)\|e_L^{j}\|_{E^{j}} + C\delta_{0}(\lambda_1) \|  e_L^{j}\|_{E^{j}}\\
      & \leq \left\{1- \alpha^{j}(1-\theta)+C\delta_0(\lambda_1)\right\}\|{e}_L^{j}\|_{E^{j}}.
    \end{aligned}
  \end{equation*}
 % where $c(h_0)=(1 + \frac{C\delta_0(\lambda_1)}{(1-\alpha^j(1-\theta))})$.  
 Since $\lambda^{j}=Rq(u^{j})$ and $\|u^{j}\|_0=1$, we get
\begin{equation}\label{eljtou1j}
  \|{e}_L^{j}\|_{E^{j}}^2 = (\lambda^{j}-\lambda_{1,L}) \|Q_1^L u^{j}\|_0^2 \leq \lambda^{j}-\lambda_{1,L}.
\end{equation}
Noting that $U^{j} + $span$\{t^{j+1}\} \subset W^{j+1},$ we know $$\lambda^{j+1}\leq \tilde{\lambda}^{j+1}=Rq (\tilde{u}^{j+1}),$$
which, together with  \eqref{eljtou1j}, yields
$$\lambda^{j+1} -\lambda_{1,L}  \leq  {\tilde{\lambda}^{j+1}}- \lambda_{1,L} = \|\tilde{e}_L^{j+1}\|_{E^{j}}^2 + ({\lambda}^{j}-\lambda_{1,L})\|\tilde{e}_L^{j+1}\|_{0}^2.$$
Further, using Lemma \ref{ukulemma}, the Poincar\'e inequality, \eqref{lambdaueq2} and \eqref{atoej}, we have
\begin{equation}\label{lambdajtoej}
   \lambda^{j+1} -\lambda_{1,L} \leq \left(1+C\delta_{0}^2(\lambda_1)\right)\left\{1- \alpha^{j}(1-\theta)+ C\delta_{0}(\lambda_1)\right\}^2 \|{e}_L^{j}\|_{E^{j}}^2. 
\end{equation}
Let $s=1- \alpha^{j}(1-\theta)+ C\delta_{0}(\lambda_1)>0$ and $t=C\delta_{0}(\lambda_1)\left\{1- \alpha^{j}(1-\theta)+ C\delta_{0}(\lambda_1)\right\}>0.$ As $s^2+t^2\leq (s+t)^2$, we get 
\begin{equation}\notag
  \lambda^{j+1}-\lambda_{1,L} \leq \left\{1- \alpha^{j}(1-\theta)+ C\delta_{0}(\lambda_1)\right\}^2  (\lambda^{j}-\lambda_{1,L}).
\end{equation}
 Moreover, by \eqref{eljtou1j} and \eqref{lambdajtoej}, we have $$\|e_L^{j+1}\|_{E^{j}} \leq \sqrt{ {\lambda}^{j+1}- \lambda_{1,L}} \leq \left\{1- \alpha^{j}(1-\theta) + C\delta_{0}(\lambda_1)\right\}\|{e}_L^{j}\|_{E^{j}},$$
which completes the proof of this theorem.\qed

 \subsection{The estimate of $R_1^{j}$}
This subsection gives a proof of (A1). We begin by introducing an operator $G_L^{j}: U_2^L \to U_2^L$, defined as:
 $$G_L^{j}:=I+\alpha^{j} Q_2^L B_L^{j}(\lambda^{j} I-A_L)=I-\alpha^{j} Q_2^L B_L^{j} \tilde{A}_L^{j}.$$ 
Next we provide an estimate for $\|G_L^{j} e_L^{j}\|_{E^{j}}$ in Lemma \ref{gljupper2}. 
\begin{lemma}\label{gljupper2}
  For sufficiently small $h_0$, it holds that
  $$\|G_L^{j} v\|_{E^{j}}\leq\left(1-\alpha^{j} (1-\theta)\right)\|v\|_{E^{j}}\ \ \ \ \forall\ v\in U_{2}^{L},$$
  where $\theta=\sqrt{1-\frac{(2-\omega)(1-C\delta_{0}(\lambda_1))^2}{C C_h^{\rho}}+C \delta^2_0(\lambda_1)}\in (0,1)$ and $0<\alpha^{j}<\frac{1}{1-\theta}.$
\end{lemma}
\begin{proof}
  According to the definition of $G_L^{j}$ and \eqref{blmtoelj}, we have
  \begin{equation}\label{glj1}
    \|G_L^{j} v\|_{E^{j}}=\|(I-\alpha^{j} Q_2^L)v +\alpha^{j} Q_2^L E_L^{j} v \|_{E^{j}}\leq (1-\alpha^{j})\|v\|_{E^{j}} + \alpha^{j} \|Q_2^L E_L^{j} v\|_{E^{j}}.
  \end{equation}
Owing to the definition of $(\cdot,\cdot)_{E^{j}}$, \eqref{directsumonL} and \eqref{lambdaueq2}, we deduce 
\begin{equation*}
  \begin{aligned}
    (Q_2^L E_L^{j} v,Q_2^L E_L^{j} v)_{E^{j}}&=(E_L^{j} v, E_L^{j} v)_{E^{j}}-(Q_1^L E_L^{j} v,E_L^{j} v)_{E^{j}}\\
    &=(E_L^{j} v, E_L^{j} v)_{E^{j}}-(\lambda_{1,L}-\lambda^{j})b(Q_1^L E_L^{j} v,E_L^{j} v )\\
    &\leq (E_L^{j} v, E_L^{j} v)_{E^{j}} + (\lambda^{j}-\lambda_{1,L})b(E_L^{j} v, E_L^{j} v).
  \end{aligned}
\end{equation*} 
Further, by Lemma \ref{Kjlowerbound} and Lemma \ref{kjupperbound}, we get
 \begin{equation*}
 %\label{eljej}
     ( E_L^{j} v, E_L^{j} v)_{E^{j}} \leq \left\{1-\frac{(2-\omega)\left(1-C\tilde{C}_h^{\rho}\delta_{0}(\lambda_1)\eta_A(V_0)\right)^2}{C C_h^{\rho}}\right\}\|v\|_{E^{j}}^2,
 \end{equation*}                                                                                                      
 which, together with Lemma \ref{eljl2upper} and \eqref{lambdajto1l}, yields
 \begin{equation}\label{q2lelj}
  \|Q_2^L E_L^{j} v\|_{E^{j}} \leq \sqrt{1-\frac{(2-\omega)\left(1-C\tilde{C}_h^{\rho}\delta_{0}(\lambda_1)\eta_A(V_0)\right)^2}{C C_h^{\rho}}+C \delta^2_0(\lambda_1)}\ \|v\|_{E^{j}}.
 \end{equation}
 For brevity, let $\theta=\sqrt{1-\frac{(2-\omega)\left(1-C\tilde{C}_h^{\rho}\delta_{0}(\lambda_1)\eta_A(V_0)\right)^2}{C C_h^{\rho}}+C \delta^2_0(\lambda_1)}.$ 
 Combining \eqref{glj1} and \eqref{q2lelj}, we obtain
 \begin{equation}\notag
  \|G_L^{j} v \|_{E^{j}}\leq \left( 1- \alpha^{j} (1- \theta) \right) \|v\|_{E^{j}},
 \end{equation}
 here we take $0<\alpha^{j}<\frac{1}{1-\theta}$ and choose sufficiently small $h_0$ to satisfy $0<\theta<1$. 
 \qed
\end{proof}

From \eqref{elj2}, we know $R_1^{j}=G_L^{j} e_L^{j}.$ Thus, we obtain the following estimate: 
\begin{equation*}
%\label{r1jfinal}
    \|R_1^{j}\|_{E^{j}} = \|G_L^{j} e_L^{j}\|_{E^{j}} \leq\left( 1- \alpha^{j} (1- \theta) \right)\|e_L^{j}\|_{E^{j}}.
\end{equation*}

\subsection{The estimate of $R_2^{j}$}
\par In this subsection, we provide a rigorous proof of (A2). From \eqref{elj2}, we deduce 
\begin{equation}\notag
\begin{aligned}
       R_2^{j}&=\alpha^{j} Q_2^L B_L^{j} \tilde{A}_L^{j} Q_1^L u^{j}+ \alpha^{j} Q_2^L Q_{U^{j}} B_L^{j} r^{j}\\
       &=: R_{2,1}^{j} + R_{2,2}^{j}.
\end{aligned}
\end{equation}
Next we give the estimates for $R_{2,1}^{j}$ and $R_{2,2}^{j}$ in Lemma \ref{r2j1lemma}. 
\begin{lemma}\label{r2j1lemma}
   It holds that
   \begin{equation}\label{r21final}
   \|R_{2,1}^{j}\|_{E^{j}} \leq C \delta_0(\lambda_1) \|e_L^{j}\|_{E^{j}},
   \end{equation}
   and
   \begin{equation}\label{r22final}
               % \|R_{2,1}^{j}\|_{E^{j}} &\leq C \delta_0(\lambda_1) \|e_L^{j}\|_{E^{j}},\label{r21final} \\
        \|R_{2,2}^{j}\|_{E^{j}} \leq C\delta_0(\lambda_1) \|e_L^{j}\|_{E^{j}}.
   \end{equation}
\end{lemma}
\begin{proof}
      Firstly, using \eqref{atoej}, \eqref{blmtoelj}, \eqref{blmsymmetry} and Lemma \ref{eljl2upper}, we get
  \begin{equation}\notag
  \begin{aligned}
   \|R_{2,1}^{j}\|_{E^{j}}^2 &= (\alpha^{j})^2
     \|Q_2^L B_L^{j} \tilde{A}_L^{j} Q_1^L u^{j}\|_{E^{j}}^2= (\alpha^{j})^2
     \|Q_2^L (I-B_L^{j} \tilde{A}_L^{j}) Q_1^L u^{j}\|_{E^{j}}^2\\
     &=(\alpha^{j})^2 ( E_L^{j} Q_1^L u^{j},Q_2^L E_L^{j} Q_1^L u^{j})_{E^{j}}= (\alpha^{j})^2 b(\tilde{A}_L^{j} Q_1^L u^{j}, \tilde{E}_L^{j} Q_2^L E_L^{j} Q_1^L u^{j})\\
     & \leq C \|\tilde{A}_L^{j} Q_1^L u^{j}\|_0 \|Q_2^L E_L^{j} Q_1^L u^{j}\|_{E^{j}},
  \end{aligned}
  \end{equation}
  which, together with \eqref{lambdajto1l} and \eqref{eljtou1j}, completes the proof of \eqref{r21final}.

Further, due to Lemma \ref{eljl2upper}, \eqref{lambdajto1l}, {\eqref{blmtoelj}} and \eqref{eljtou1j}, we have
\begin{equation}\label{r2j1}
\begin{aligned}
     \|\alpha^{j} Q_2^L Q_{U^j} B_L^{j} r^{j}\|_{E^{j}}&=\alpha^{j} \|Q_2^L b(B_L^{j} r^{j},u^{j})u^{j}\|_{E^{j}}=\alpha^{j} |b(B_L^{j} r^{j},u^{j})|\|e_L^{j}\|_{E^{j}}\\
     &\leq \alpha^{j} \sqrt{\lambda^{j}-\lambda_{1,L}}\|B_L^{j} \tilde{A}_L^{j} u^{j}\|_0\\
       & \leq \alpha^{j} \sqrt{\lambda^{j}-\lambda_{1,L}} \|(I- E_L^{j})Q_1^L u^{j}\|_0 +\alpha^{j} \sqrt{\lambda^{j}-\lambda_{1,L}}\|(I- E_L^{j})Q_2^L u^{j}\|_0\\
       & \leq  \alpha^{j} \sqrt{\lambda^{j}-\lambda_{1,L}} \|(I- E_L^{j})Q_1^L u^{j}\|_0 +C \delta_0(\lambda_1) \|e_L^{j}\|_{E^{j}}.
\end{aligned}
\end{equation}
It follows from \eqref{eljsum} that
\begin{equation}\label{elju1l}
  \|(I- E_L^{j})Q_1^L u^{j}\|_0 =\|\sum_{l=0}^L K_l^{j} E_{l-1}^{j} Q_1^L u^{j}\|_0 \leq \|K_0^{j} Q_1^L u^{j}\|_0 + \| \sum_{l=1}^L K_l^{j} E_{l-1}^{j} Q_1^L u^{j}\|_0.
\end{equation}
By the definition of $K_0^{j}$, we have
\begin{equation}
  \|K_0^{j} Q_1^L u^{j}\|_0 =\| (\tilde{A}_0^{j})^{-1} Q_2^0 Q_0 \tilde{A}_L^{j} Q_1^L u^{j}\|_0 \leq C \|\tilde{A}_L^{j} Q_1^L u^{j}\|_0=C(\lambda^{j}-\lambda_{1,L})\|Q_1^L u^{j}\|_0 .
\end{equation}
{Taking $w_L=\sum_{l=1}^L K_l^{j} E_{l-1}^{j} Q_1^L u^{j}$ and $v_L=Q_1^L u^{j}$ in Lemma \ref{klfl1} and using Lemma \ref{eljl2upper}}, we get
\begin{equation}\label{n1_}
\begin{aligned}
     \|\sum_{l=1}^L K_l^{j} E_{l-1}^{j} Q_1^L u^{j}\|_{0} & \leq C  \left((Q_1^L u^{j},Q_1^L u^{j})_{E^{j}}+ (\lambda^{j}-\lambda_{1,L}) \|E_L^{j} Q_1^L u^{j}\|_0^2 \right)^{1/2}\\
     & \leq  C \delta_0 (\lambda_1)\|Q_1^L u^{j}\|_0.
\end{aligned}
\end{equation}
Finally, substituting \eqref{elju1l}$\sim$\eqref{n1_} into \eqref{r2j1}, we may obtain \eqref{r22final}.
\end{proof}

\section{Numerical experiments}
In this section, we present several examples to illustrate the optimality of our method, including the uniform convergence rate with respect to degrees of freedom and the optimal complexity $O(N)$. This paper focuses on eigenvalue problems with singularity in three specific cases: (\uppercase\expandafter{\romannumeral1}) the L-shaped domain, (\uppercase\expandafter{\romannumeral2}) the crack domain, (\uppercase\expandafter{\romannumeral3}) highly discontinuous coefficients. We use the adaptive $P_1$-conforming finite element and the newest vertex bisection strategy to solve the principal eigenpair of the second order elliptic operator in numerical experiments. When $\rho_{\max}\gg 1=\rho_{\min}$, we use the local error estimator $\eta_{L}(u_{1}^{ L}, E)$ defined in Subsection 6.3. When $\rho(x) \equiv 1$, we  utilize the local error estimator defined in \cite{dai2008convergence}. To make its definition clear, we first introduce two quantities: 
$$
\begin{aligned}
\mathcal{R}_T\left(u_{1}^{L}\right) & :=\lambda_{1}^L u_{1}^L+\nabla \cdot\left( \nabla u_{1}^L\right) \quad \text { in } T \in \mathcal{T}_L, \\
J_E\left(u_{1}^L\right) & :=-(\nabla u_{1}^L) ^{+} \cdot \nu^{+}- (\nabla u_{1}^L)^{-} \cdot \nu ^{-}:=\left[\left[ \nabla u_{1}^L\right]\right]_E \cdot v_E \quad \text { on } E \in \mathcal{E}_L,
\end{aligned}
$$
where $E$ is the common side of elements $T^+$ and $T^-$  with unit outward normals $\nu^{+}$ 
and $\nu^{-}$, $\nu_{E}$ = $\nu^{-}$ and $\mathcal{E}_L$ represents the set of interior edges of $\mathcal{T}_L$.
The local error estimator is defined as follows:
$$
\eta_L^2\left(u_{1}^L, T\right):=h_T^2\left\|\mathcal{R}_T\left(u_{1}^L\right)\right\|_{0, T}^2+\sum_{E \in \mathcal{E}_L, E \subset \partial T} h_E\left\|J_E\left(u_{1}^L\right)\right\|_{0, E}^2,
$$
where $h_T$ is the diameter of $T$.
Algorithm 1 terminates iteratively when
$$ stop.=|\lambda^{j+1}-\lambda^{j}|\leq tol = 1e^{-10}.$$
%For the stopping criterion of the algorithms, we choose the
\subsection{L-shaped domain}
In this subsection, we present numerical results for the Laplacian eigenvalue problem ($\rho(x)\equiv 1$) on the L-shaped domain with a local Jacobi smoother parameter $\gamma=0.8$.
\begin{example}
  We consider \eqref{modelproblem} on L-shaped domain $\Omega=(-1,1)\times (-1,1)\backslash [0,1) \times  ( -1,0]$. 
\end{example}
%\textwidth
\begin{figure}[H]
  \centering
\includegraphics[width=0.45\textwidth]
  {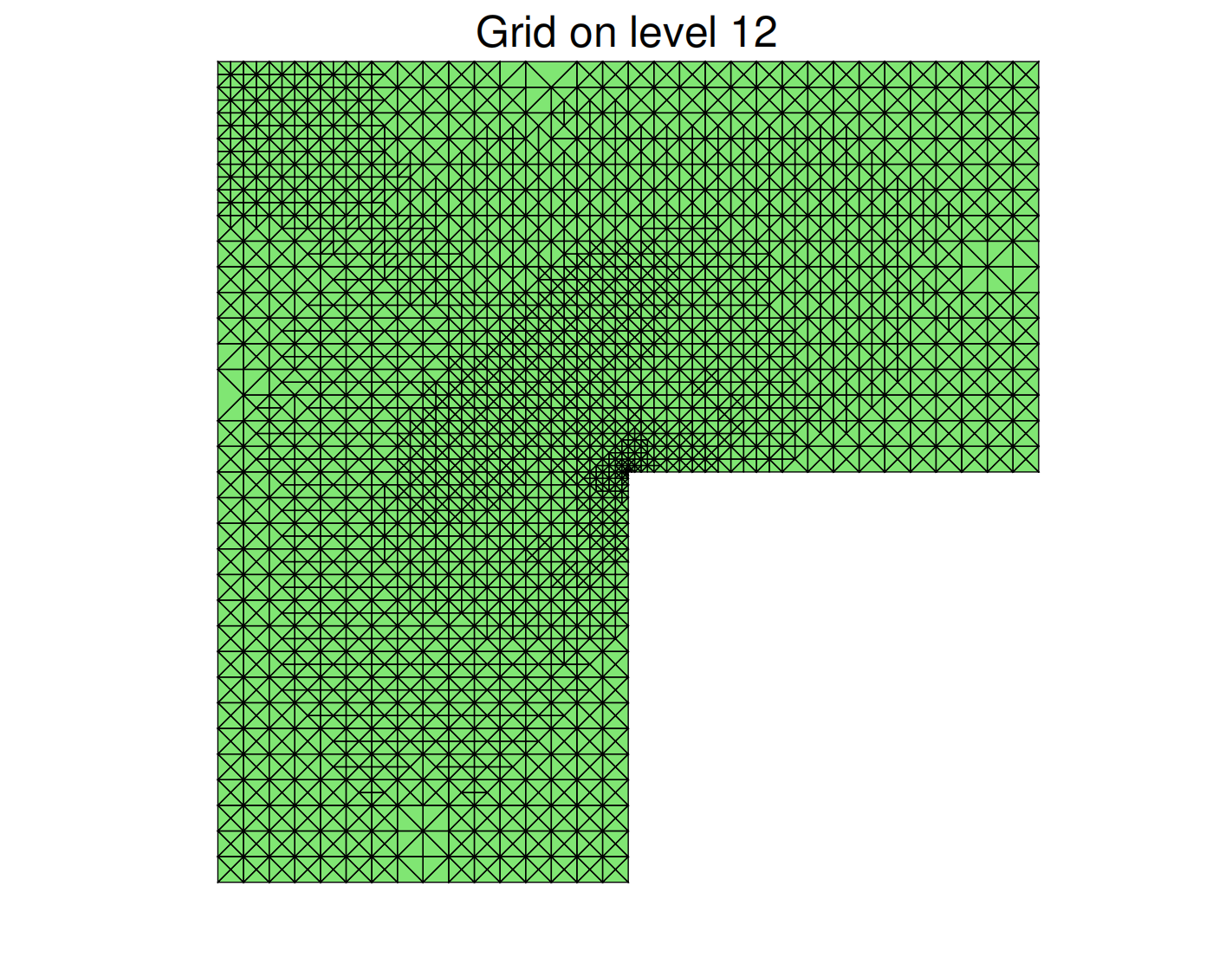}
  \caption{The local refined mesh on adaptive level 12}\label{figure1}
\end{figure}
\par  We first introduce some notations: $d.o.f.$ denotes degrees of freedom; $it.$ denotes the number of iterations; $\lambda_1^k$ denotes the final iterative solution at level $k$ obtained by our iterative algorithm; $\eta$ denotes the a posteriori estimator and  $stop.$ denotes the value of $|\lambda^{j+1}-\lambda^{j}|$ when our iterative algorithm
stops. The mesh obtained after 12 adaptive iterations is shown in Figure \ref{figure1}. Numerical results in Table \ref{table1} demonstrate that the number of iterations  does not increase with increasing mesh levels and degrees of freedom, which indicates that the convergence rate is optimal. The curves are parallel in Figure \ref{figure2}, which shows that the iterative solution (on each level) converges to the discrete solution of \eqref{discretevariational}. Furthermore, the curve in Figure \ref{figure3} illustrates that the algorithm has optimal computational complexity $O(N)$.
\begin{table}
  \centering
    \caption{The number of iterations on each level}
    \label{table1}
  \begin{tabular}{ c c  c c c}
  \toprule
  $level\ (k)$ & $d.o.f$ & $it.$ & $stop.$ & $\lambda_1^{k}$ \\
  \midrule
  20 & 23346 &6& 5.1114e-11& 9.641716304 \\
  \midrule
  24 & 61979 & 6& 6.2871e-11&9.640489992 \\
   \midrule
  28 & 158501 & 6& 8.5283e-11&9.640003301 \\
 \midrule
  32 & 414279 &5 &5.3108e-11 &9.639835271 \\
  \midrule
  36 & 1107619 &5 & 1.6225e-11&9.639767277 \\
  \midrule
  & & & \\ [-9pt]
  40 & 2877315 & 5&5.0788e-11 &9.639739416  \\
  \botrule
  \end{tabular}
\end{table}
\begin{figure}[h]
  \begin{minipage}[t]{0.5\textwidth}
    \centering
\includegraphics[width=0.6\textwidth]{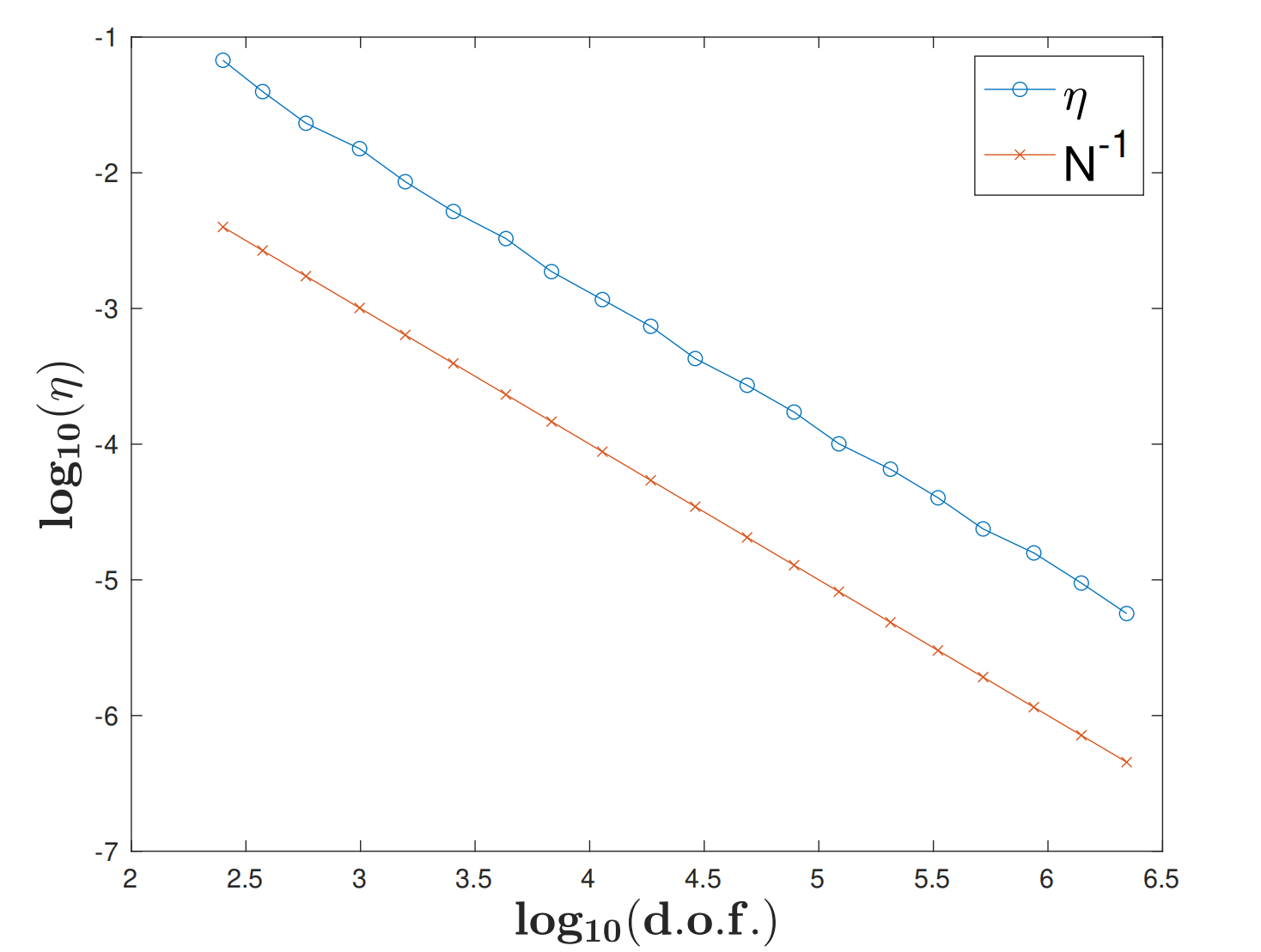}
    \caption{A posteriori estimator }\label{figure2}
  \end{minipage}
  \begin{minipage}[t]{0.5\textwidth}
    \centering
\includegraphics[width=0.6\textwidth]{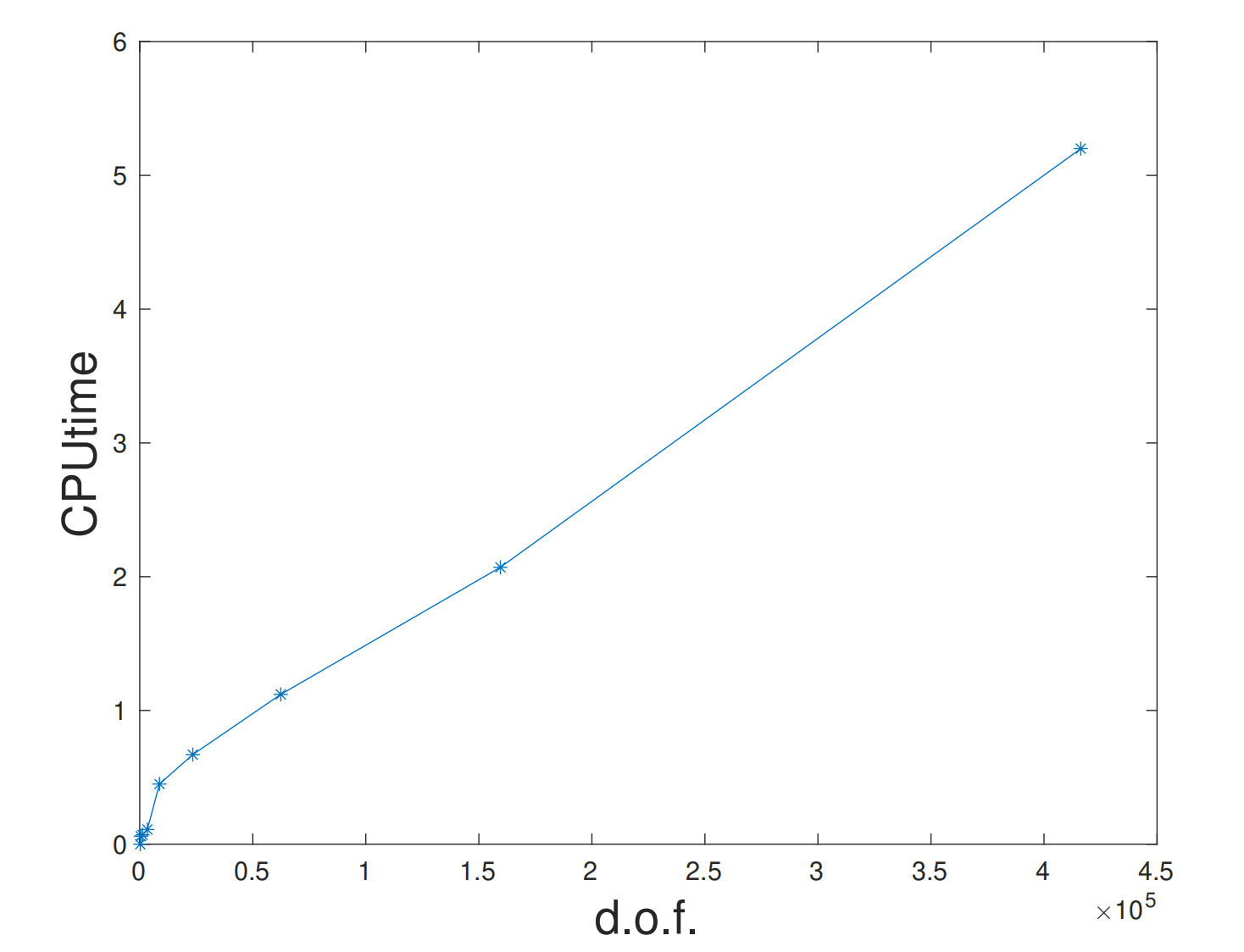}
    \caption{CPU time} \label{figure3}
  \end{minipage}
  \end{figure}
  \subsection{Crack domain}
This subsection presents numerical results for the Laplacian eigenvalue problem ($\rho(x)\equiv 1$) on the crack domain with a local Jacobi smoother parameter $\gamma=0.8$.
\begin{example}
  We consider \eqref{modelproblem} on crack domain $\Omega=(-1,1)^2 \backslash [0,1] \times \{0\}$.
\end{example}
\begin{figure}[!htb]
  \centering
\includegraphics[width=0.4\textwidth]{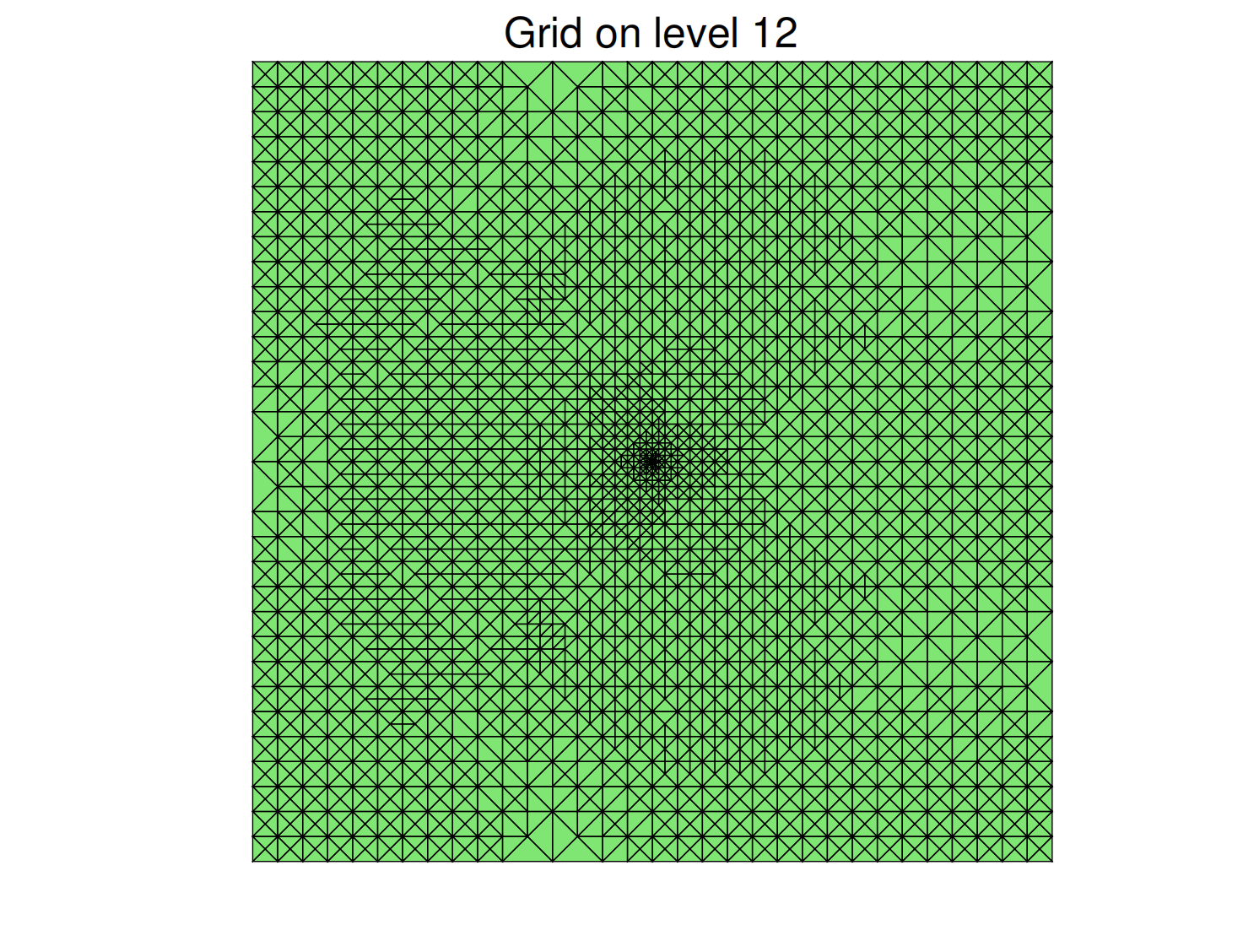}
  \caption{The local refined mesh on adaptive level 12}\label{figure4}
\end{figure}
\par Figure \ref{figure4} presents the locally refined mesh obtained after 12 adaptive steps. The numerical results presented in Table \ref{table2} indicate that the number of iterations remains stable as mesh levels and degrees of freedom increase. This suggests that our algorithm exhibits a uniform convergence rate with respect to mesh levels and degrees of freedom. The curves are parallel in Figure \ref{figure5}, which shows that the iterative solution (on each level) converges to the discrete solution of \eqref{discretevariational} in the crack case. Moreover, it is seen from Figure \ref{figure6} that the algorithm has an optimal computational complexity $O(N)$. These results are consistent with our theoretical findings.
\begin{table}[htbp]
  \centering
  \setlength{\tabcolsep}{5mm}{
    \caption{The number of iterations on each level}
    \label{table2}
        \renewcommand{\arraystretch}{1}
  \begin{tabular}{ c c c  c c }
  \toprule
  $level\ (k)$ & $d.o.f$ & $it.$ & $stop.$ &$\lambda_1^{k}$ \\
  \midrule
  20 & 24200 &7 & 5.6666e-11& 8.375251473 \\
  \midrule
  24 & 62699 & 7& 5.0886e-11 & 8.372775444 \\
   \midrule
  28 & 159500 & 7&8.1155e-11 & 8.371945433 \\
  \midrule
  32 & 422886 & 7& 3.1701e-11& 8.371578378 \\
  \midrule
 % & & & \\ [-10pt]
  36 & 1109565 & 6& 7.6005e-11& 8.371427991 \\
  \midrule
  & & & \\ [-9pt]
  40 & 2762819 & 6& 6.1393e-11 &  8.371370412  \\
  \botrule
  \end{tabular}}
\end{table}
\begin{figure}[H]
  \begin{minipage}[t]{0.5\textwidth}
    \centering
    \includegraphics[width=0.6\textwidth]{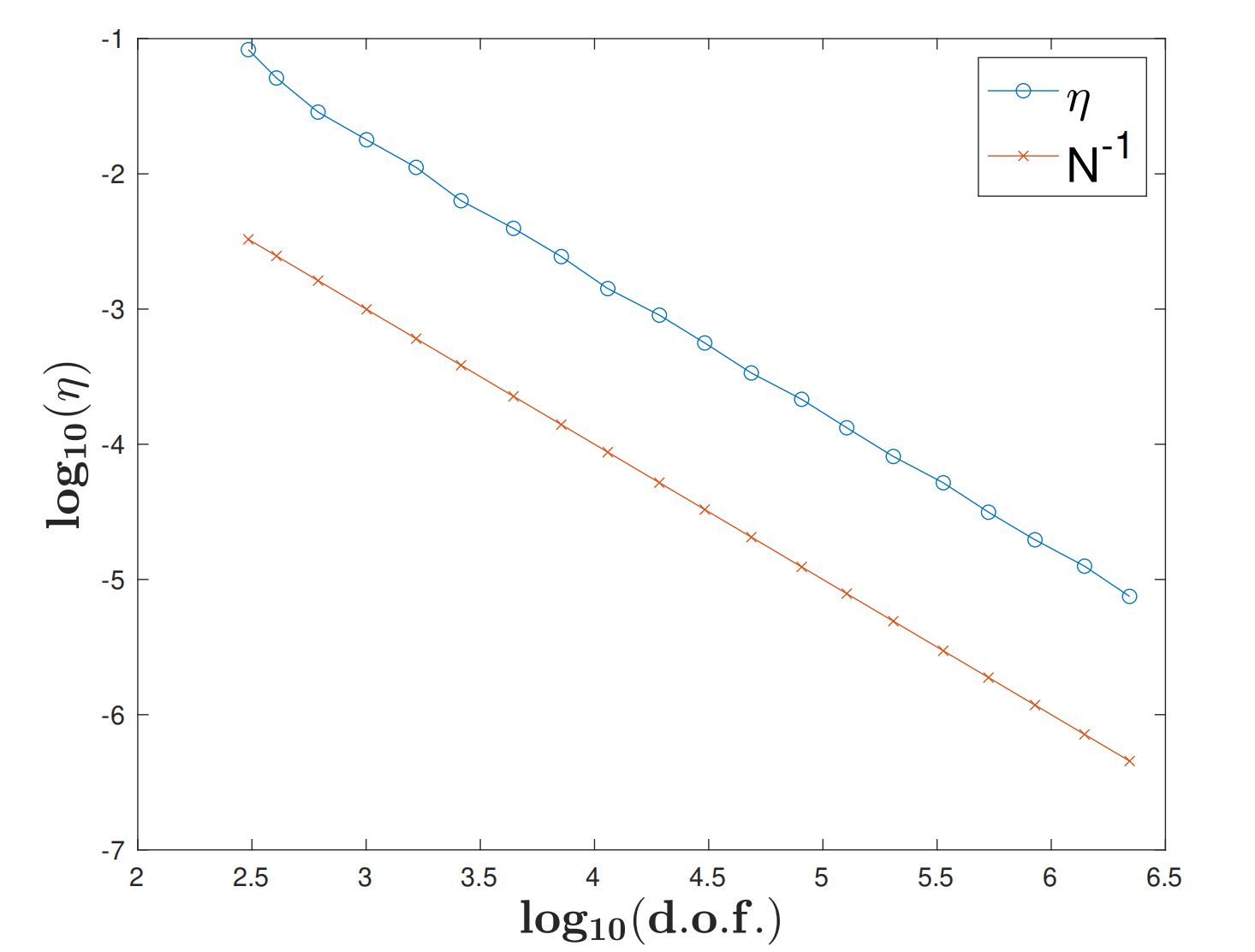}
    \caption{A posteriori estimator}\label{figure5}
  \end{minipage}
  \begin{minipage}[t]{0.5\textwidth}
    \centering
\includegraphics[width=0.6\textwidth]{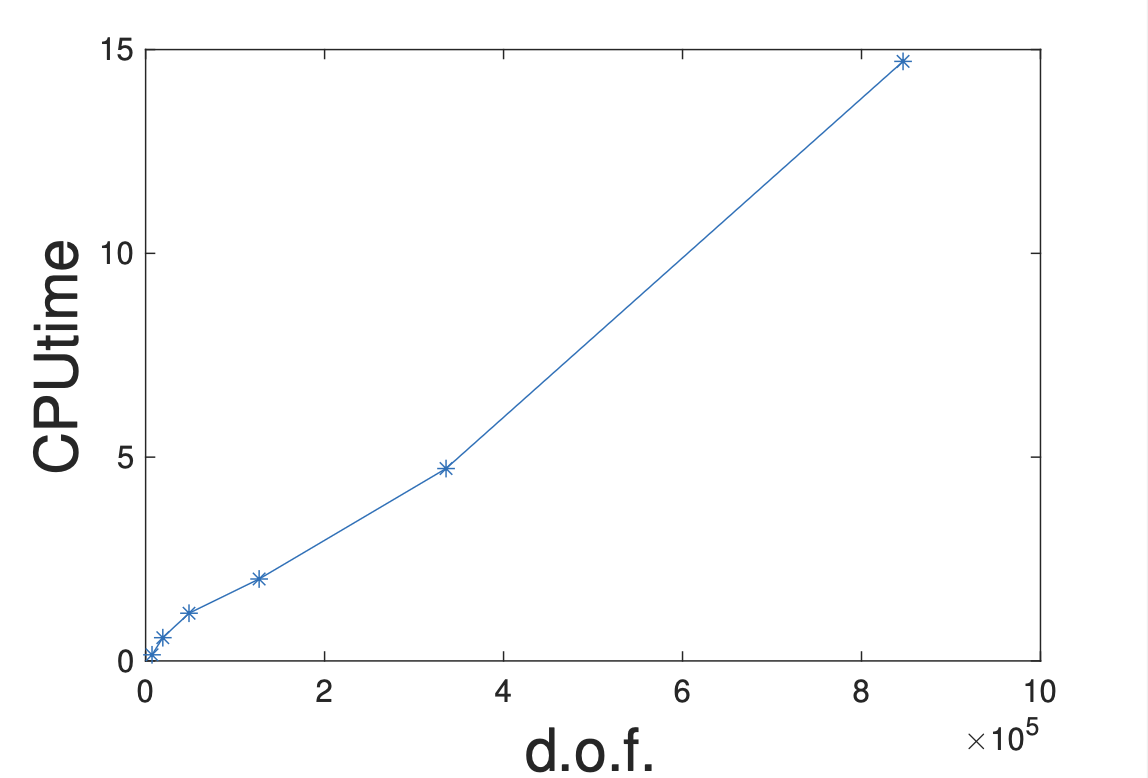}
    \caption{CPU time} \label{figure6}
  \end{minipage}
  \end{figure}
  
  \subsection{Highly discontinuous coefficients eigenvalue problem}
\begin{example}
  We consider the highly discontinuous coefficients eigenvalue problem on domain $\Omega=(0,1)^2$,
\begin{equation}\label{jumpproblem}
   \begin{cases}
      -\nabla\cdot (\rho(x)\nabla u )=\lambda u\ \ \ \ &\text{in} \ \ \Omega,\\
     \ \ \ \ \ \ \ \ \ \ \ \ \  \ u=0\ \ \ \ & \text{on} \ \ \partial \Omega,
   \end{cases}
\end{equation}
where $\rho(x)$ is a piecewise constant function
\begin{equation}\notag
\rho(x)=
    \begin{cases}
       \ \ \  1 \ \   &\text{in} \ \Omega_1\ \text{or}\ \Omega_4,\\
        \mu \gg 1 \ \ \ \   &\text{in}\ \Omega_2\ \text{or}\ \Omega_3, 
    \end{cases}
\end{equation}
where $\Omega_1=(0,0.5)\times(0.5,1),$ $\Omega_2=[0.5,1)\times[0.5,1),$ $\Omega_3=(0,0.5]\times(0,0.5]$ and $\Omega_4 = (0.5,1)\times(0,0.5)$ in Figure \ref{figure51}.
\end{example}
\par In this subsection, we utilize the maximum marking strategy \cite{MR745088} to solve \eqref{jumpproblem}. To clarify the local error estimator \cite{czm2002} used in this subsection, we introduce some notations. For any $T \in \mathcal{T}_L$ and $E \in \mathcal{E}_L$, let
$$
\omega_T=\cup\left\{T^{\prime} \in \mathcal{T}_L: \bar{T}^{\prime} \cap \bar{T} \neq \emptyset\right\}, \quad \omega_E=\cup\left\{T^{\prime} \in \mathcal{T}_L: \bar{T}^{\prime} \cap \bar{E} \neq \emptyset\right\} .
$$
Further, we define
$$
\Lambda_T:= \begin{cases}\max _{T^{\prime} \in \omega_T}\left(\frac{\rho_T}{\rho_{T^{\prime}}}\right) & \text { if } T \text { has one singular node, } \\ 1 & \text { otherwise, }\end{cases}
$$
where $\rho_T$ is the constant value of $\rho(x)$ on $T \in \mathcal{T}_L$.
For any $E \in \mathcal{E}_L$, denote by $\Omega_E$ the collection of two elements sharing the side $E$. Let $\rho_E=\max_{T \in \Omega _E} (\rho_T).$ The local error estimator associated with $E \in \mathcal{E}_L$
is defined as:
$$
\eta_L^2(u_1^{L},E)=\sum_{T \in \Omega_E} \Lambda_T \left\|h_T  \rho_T^{-1 / 2} \lambda_{1}^{L} u_{1}^{L} \right\|_{L^2 (T)}^2+\Lambda_E\left\|h_E^{1/2} \rho_E^{-1/2} J_E(u_1^{L})\right\|_{L^2 (E)}^2 \quad \text { on } E \in \mathcal{E}_L,
$$ 
where $\Lambda_E=\max_{T\in \Omega_E} (\Lambda_T).$
As shown in Figure \ref{figure52}, the mesh is refined near the stripes where the coefficient function $\rho(x)$ changes. From Table \ref{table5}, the number of iterations remains stable as both mesh levels and degrees of freedom increase. Furthermore, the number of iterations also remains stable with the discontinuous coefficients increasing. The above phenomena show that our method has a uniform convergence rate with respect to mesh levels, degrees of freedom, and discontinuous coefficients. 
Figure \ref{figure3eta} shows that the iterative solution (on each level) converges to the discrete solution of \eqref{jumpproblem}. Additionally, the curve in Figure \ref{figure3cpu} illustrates that the computational complexity of our algorithm is optimal $O(N)$ even with the discontinuous coefficients.
 \begin{figure}[h]
  \begin{minipage}[t]{0.5\textwidth}
    \centering
\includegraphics[width=0.65\textwidth]{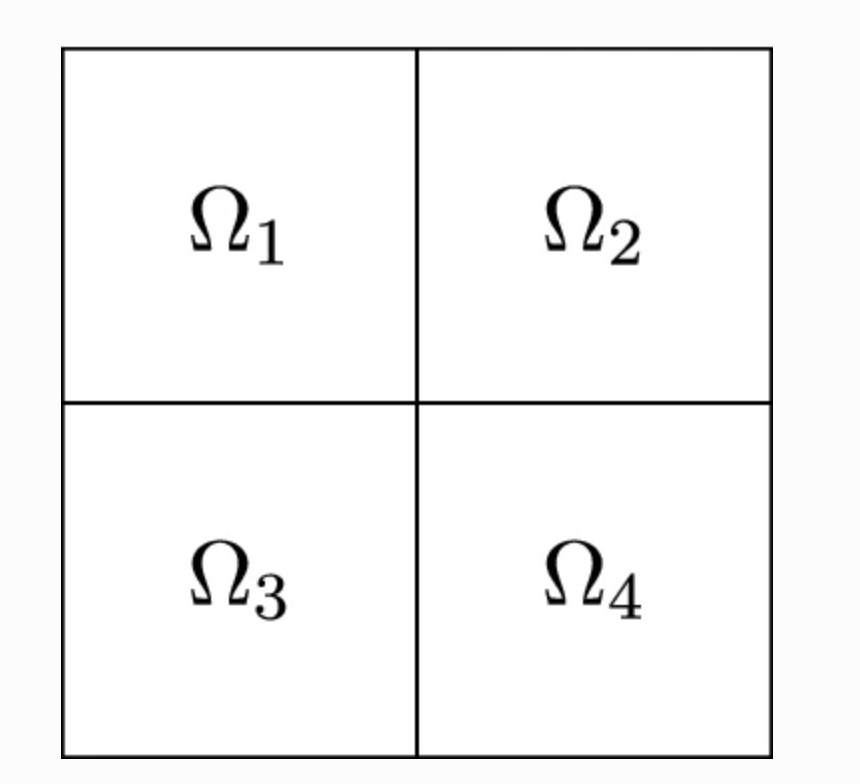}
    \caption{The definition of $\Omega_i,\ i=1,2,3,4$ }\label{figure51}
  \end{minipage}
  \begin{minipage}[t]{0.5\textwidth}
    \centering
\includegraphics[width=0.85\textwidth]{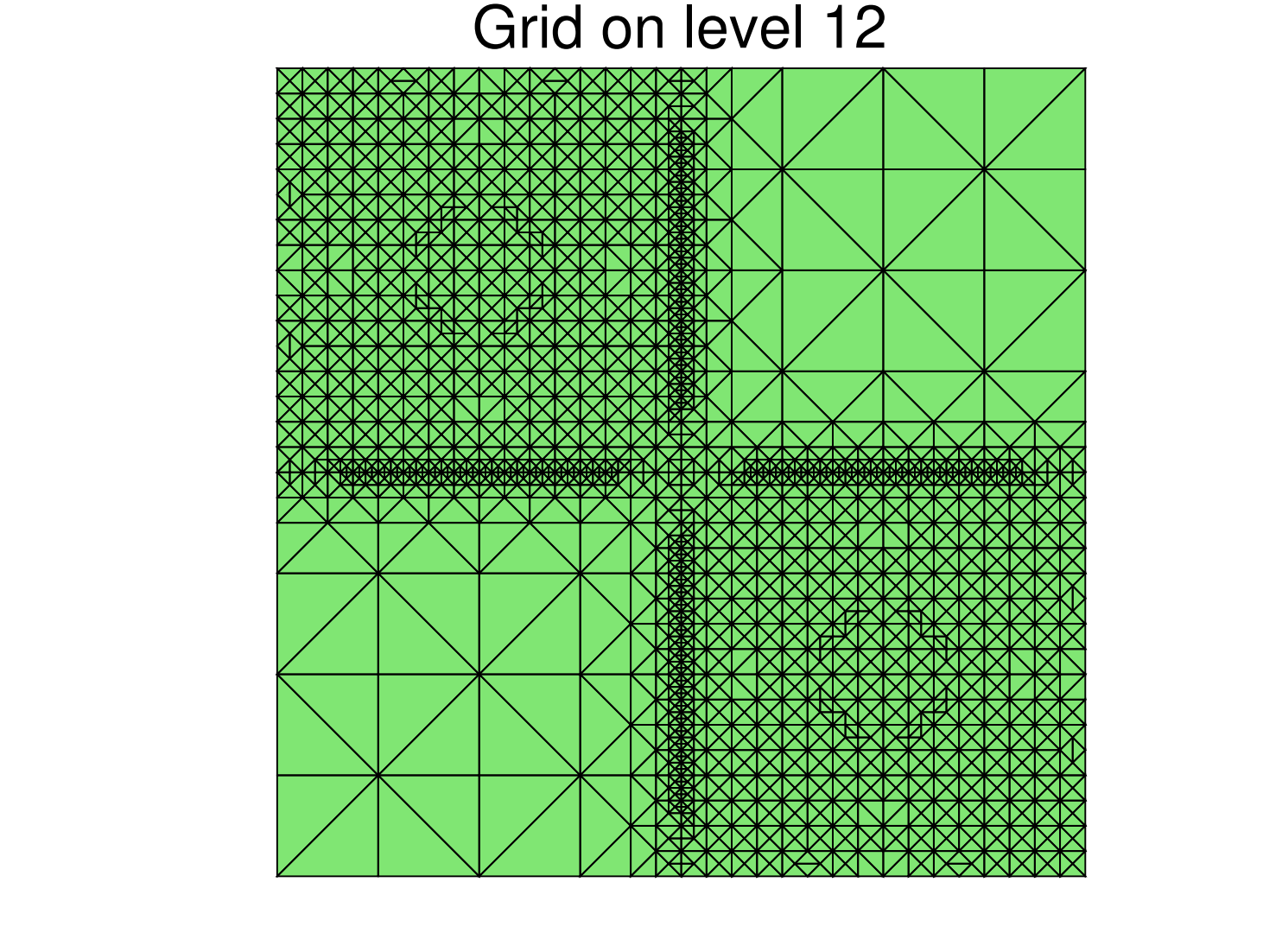}
    \caption{The local refined mesh on adaptive level 12 with $\mu=10^8$} \label{figure52}
  \end{minipage}
  \end{figure}
  
\begin{figure}[H]
  \begin{minipage}[t]{0.5\textwidth}
    \centering
\includegraphics[width=0.6\textwidth]{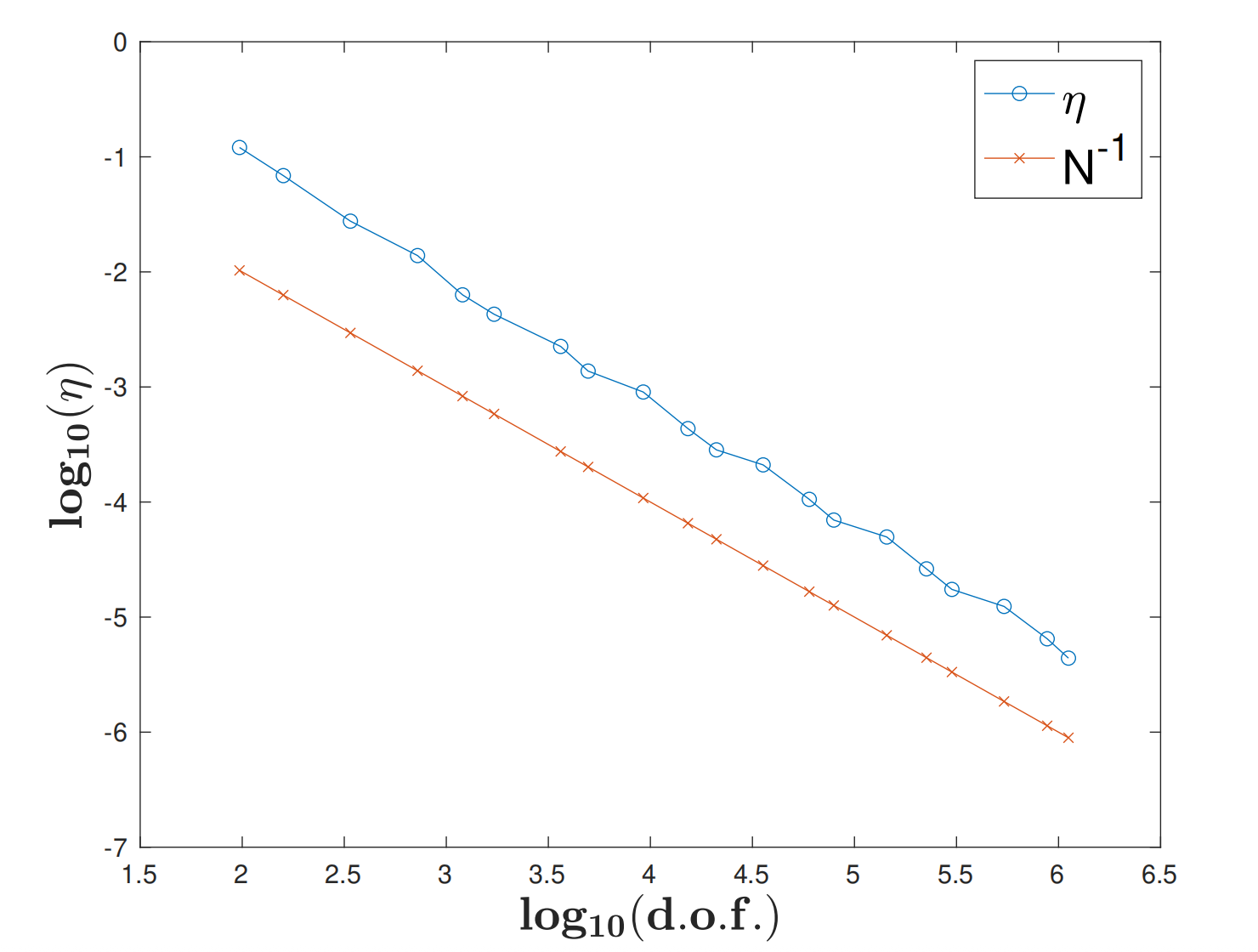}
    \caption{A posteriori estimator with $\mu=10^8$}\label{figure3eta}
  \end{minipage}
  \begin{minipage}[t]{0.5\textwidth}
    \centering
\includegraphics[width=0.6\textwidth]{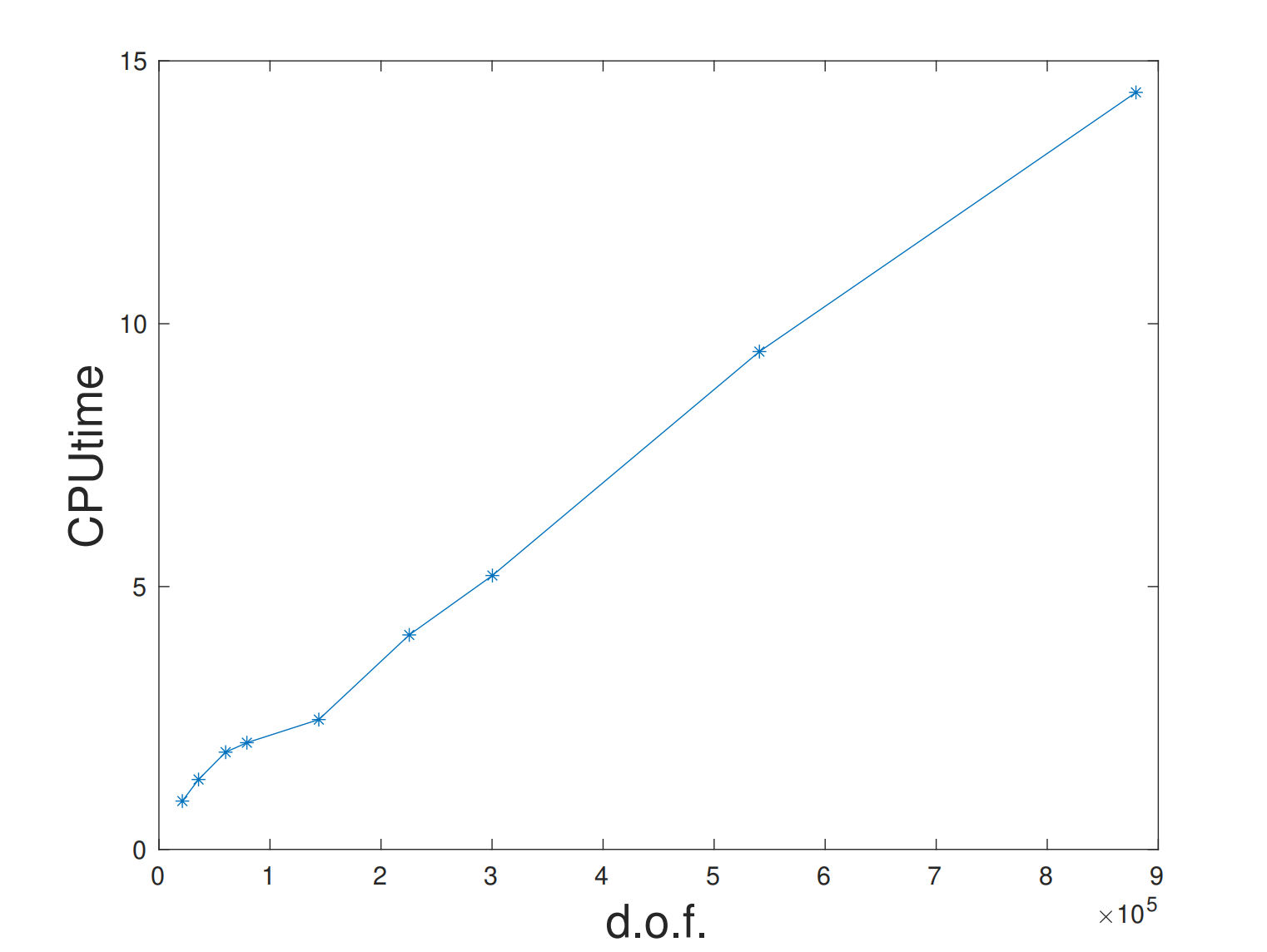}
    \caption{CPU time with $\mu=10^8$} \label{figure3cpu}
  \end{minipage}
  \end{figure}
  
 \begin{table}[h]
    \centering
    \caption{The number of iterations of different discontinuous coefficients cases}
    \label{table5}
\begin{tabular}{c  c  c  ccccc }
\toprule \multirow{3}{*}{$\mu=10^2$} & \multicolumn{2}{c}{ Level }  & 24 & 28 & 32 & 36 & 40 \\
\cmidrule {2-8} & \multicolumn{2}{c}{ $d.o.f$ }  & 45845 & 88399 & 253465 & 657595 & 1299847 \\
 \cmidrule{2-8}&\multicolumn{2}{c}{ $it.$ }  & 8 & 8 & 7 & 7 & 6 \\
\midrule 
\multirow{3}{*}{$\mu=10^4$} & \multicolumn{2}{c}{ Level }  & 24 & 28 & 32 & 36 & 40 \\
\cmidrule {2-8} & \multicolumn{2}{c}{ $d.o.f$ }  & 48667 & 96467 & 259453 & 719585 & 1393497 \\
 \cmidrule {2-8} &\multicolumn{2}{c}{ $it.$ }  & 7 & 7 & 7 & 6 & 6 \\
 \midrule
\multirow{3}{*}{$\mu=10^6$} & \multicolumn{2}{c}{ Level }  & 24 & 28 & 32 & 36 & 40 \\
\cmidrule{2-8} & \multicolumn{2}{c}{ $d.o.f$ }  & 48667 & 96497 & 259521 & 719807 & 1396113 \\
 \cmidrule {2-8} &\multicolumn{2}{c}{ $it.$ }  & 7 & 7 & 6 & 6 & 5 \\
\midrule 
\multirow{3}{*}{$\mu=10^8$} & \multicolumn{2}{c}{ Level }  & 24 & 28 & 32 & 36 & 40 \\
\cmidrule {2-8} & \multicolumn{2}{c}{ $d.o.f$ }  & 48667 & 96473 & 259521 & 719807 & 1396153 \\
 \cmidrule{2-8} &\multicolumn{2}{c}{ $it.$ }  & 7 & 7 & 6 & 6 & 5 \\
 \bottomrule
\end{tabular}
\end{table}

\section{Conclusion}\label{sec7}
Based on the local multilevel method, we proposed a new preconditioned Jacobi-Davidson method for solving elliptic eigenvalue problems on adaptive meshes. The proposed method may be used to efficiently address the singularity of the first eigenfunction. It is proven that the convergence factor is optimal with respect to degrees of freedom and mesh levels. Numerical results demonstrate that our method accurately captures the locations of singularities while achieving the optimal computational complexity $O(N)$. 

\section*{Declarations}

\begin{itemize}
\item \textbf{Funding} The second author is supported by National Natural Science Foundation of China (Grant No. 12401481). The third author is supported by National Natural Science Foundation of China (Grant No. 12331015).

\item \textbf{Conflict of interest/Competing interests} There is no potential conflict of interest.

\item \textbf{Ethics approval and consent to participate} This article does not contain any studies involving animals or human participants.

\item \textbf{Consent for publication} Not applicable.

\item \textbf{Data availability} This research does not use any external or author-collected data.

\item \textbf{Materials availability} Not applicable.

\item \textbf{Code availability} Not applicable.

\item \textbf{Author contributions} All authors contributed equally.
\end{itemize}

\bibliography{sn-bibliography}

@article{babuvska1989finite,
	author = {Babu\v{s}ka, I. and Osborn, J. E.},
	doi = {10.2307/2008468},
	fjournal = {Mathematics of Computation},
	issn = {0025-5718},
	journal = {Math. Comp.},
	mrclass = {65N30 (65N25)},
	mrnumber = {962210},
	mrreviewer = {Randall J. LeVeque},
	number = {186},
	pages = {275--297},
	title = {Finite element-{G}alerkin approximation of the eigenvalues and eigenvectors of selfadjoint problems},
	url = {https://doi.org/10.2307/2008468},
	volume = {52},
	year = {1989}}

@article{chen2015multigrid,
	author = {Chen, Hongtao and He, Yunhui and Li, Yu and Xie, Hehu},
	doi = {10.1007/s40879-014-0034-0},
	fjournal = {European Journal of Mathematics},
	issn = {2199-675X},
	journal = {Eur. J. Math.},
	mrclass = {65N30 (65N25 65N55)},
	mrnumber = {3386235},
	mrreviewer = {Steve Wright},
	number = {1},
	pages = {207--228},
	title = {A multigrid method for eigenvalue problems based on shifted-inverse power technique},
	url = {https://doi.org/10.1007/s40879-014-0034-0},
	volume = {1},
	year = {2015}}

@book{bramble2019multigrid,
    AUTHOR = {Bramble, James H.},
     TITLE = {Multigrid methods},
    SERIES = {Pitman Research Notes in Mathematics Series},
    VOLUME = {294},
    ADDRESS = {Harlow}, 
    DOI = {10.1201/9780203746332},   
    PUBLISHER = {Longman Scientific \& Technical},  
      YEAR = {1993},
     PAGES = {viii+161},
      ISBN = {0-582-23435-2},
       NOTE = {Copublished in the United States with John Wiley \& Sons, Inc., New York},  
   MRCLASS = {65-02 (65M55 65N55)},
  MRNUMBER = {1247694},
MRREVIEWER = {Mark Ainsworth},
}

@book {MR2104179,
    AUTHOR = {Toselli, Andrea and Widlund, Olof},
     TITLE = {{D}omain {D}ecomposition {M}ethods-{A}lgorithms and {T}heory},
    SERIES = {Springer Series in Computational Mathematics},
    VOLUME = {34},
ADDRESS = {Berlin},
 PUBLISHER = {Springer-Verlag},
      YEAR = {2005},
     PAGES = {xvi+450},
      ISBN = {3-540-20696-5},
   MRCLASS = {65-02 (65N55 74S05 76M10)},
  MRNUMBER = {2104179},
MRREVIEWER = {R\'{e}mi Vaillancourt},
       DOI = {10.1007/b137868},
       
}

@article{mitchell1992optimal,
	author = {Mitchell, William F.},
	doi = {10.1137/0913009},
	fjournal = {Society for Industrial and Applied Mathematics. Journal on Scientific and Statistical Computing},
	issn = {0196-5204},
	journal = {SIAM J. Sci. Statist. Comput.},
	mrclass = {65N50 (65F10 65N30)},
	mrnumber = {1145181},
	number = {1},
	pages = {146--167},
	title = {Optimal multilevel iterative methods for adaptive grids},
	url = {https://doi.org/10.1137/0913009},
	volume = {13},
	year = {1992}}

@article{dai2008convergence,
	author = {Dai, Xiaoying and Xu, Jinchao and Zhou, Aihui},
	doi = {10.1007/s00211-008-0169-3},
	fjournal = {Numerische Mathematik},
	issn = {0029-599X},
	journal = {Numer. Math.},
	mrclass = {65N25 (65N30)},
	mrnumber = {2430983},
	number = {3},
	pages = {313--355},
	title = {Convergence and optimal complexity of adaptive finite element eigenvalue computations},
	url = {https://doi.org/10.1007/s00211-008-0169-3},
	volume = {110},
	year = {2008}}

@article {liang2022two,
    AUTHOR = {Liang, Qigang and Wang, Wei and Xu, Xuejun},
     TITLE = {A two-level block preconditioned {J}acobi-{D}avidson method
              for multiple and clustered eigenvalues of elliptic operators},
   JOURNAL = {SIAM J. Numer. Anal.},
  FJOURNAL = {SIAM Journal on Numerical Analysis},
    VOLUME = {62},
      YEAR = {2024},
    NUMBER = {2},
     PAGES = {998--1019},
      ISSN = {0036-1429},
   MRCLASS = {65N55 (65N25 65N30)},
  MRNUMBER = {4735245},
       DOI = {10.1137/23M1580711},
       URL = {https://doi.org/10.1137/23M1580711},
}

@article{wu2006uniform,
	author = {Wu, Haijun and Chen, Zhiming},
	doi = {10.1007/s11425-006-2005-5},
	fjournal = {Science in China. Series A. Mathematics},
	issn = {1006-9283},
	journal = {Sci. China Ser. A},
	mrclass = {65N55},
	mrnumber = {2287269},
	mrreviewer = {Dietrich Braess},
	number = {10},
	pages = {1405--1429},
	title = {Uniform convergence of multigrid {V}-cycle on adaptively refined finite element meshes for second order elliptic problems},
	url = {https://doi.org/10.1007/s11425-006-2005-5},
	volume = {49},
	year = {2006}}

@article{xu2010optimality,
	author = {Xu, X. and Chen, H. and Hoppe, R. H. W.},
	doi = {10.1515/JNUM.2010.003},
	fjournal = {Journal of Numerical Mathematics},
	issn = {1570-2820},
	journal = {J. Numer. Math.},
	mrclass = {65N50},
	mrnumber = {2629823},
	mrreviewer = {Tsu-Fen Chen},
	number = {1},
	pages = {59--90},
	title = {Optimality of local multilevel methods on adaptively refined meshes for elliptic boundary value problems},
	url = {https://doi.org/10.1515/JNUM.2010.003},
	volume = {18},
	year = {2010}}

@article{wang2019convergence,
	author = {Wang, Wei and Xu, Xuejun},
	doi = {10.1090/mcom/3403},
	fjournal = {Mathematics of Computation},
	issn = {0025-5718},
	journal = {Math. Comp.},
	mrclass = {65N30 (65N12 65N25 65N55)},
	mrnumber = {3957894},
	mrreviewer = {Dietrich Braess},
	number = {319},
	pages = {2295--2324},
	title = {On the convergence of a two-level preconditioned {J}acobi-{D}avidson method for eigenvalue problems},
	url = {https://doi.org/10.1090/mcom/3403},
	volume = {88},
	year = {2019}}

@article{czm2002,
	author = {Chen, Zhiming and Dai, Shibin},
	doi = {10.1137/S1064827501383713},
	fjournal = {SIAM Journal on Scientific Computing},
	issn = {1064-8275},
	journal = {SIAM J. Sci. Comput.},
	mrclass = {65N50 (65N15 65N30)},
	mrnumber = {1951050},
	mrreviewer = {Piotr P. Matus},
	number = {2},
	pages = {443--462},
	title = {On the efficiency of adaptive finite element methods for elliptic problems with discontinuous coefficients},
	url = {https://doi.org/10.1137/S1064827501383713},
	volume = {24},
	year = {2002}}

@article{MR3407250,
	author = {Dai, Xiaoying and He, Lianhua and Zhou, Aihui},
	doi = {10.1093/imanum/dru059},
	fjournal = {IMA Journal of Numerical Analysis},
	issn = {0272-4979},
	journal = {IMA J. Numer. Anal.},
	mrclass = {65N25 (65N12 65N15 65N30)},
	mrnumber = {3407250},
	mrreviewer = {Severiano Gonzalez-Pinto},
	number = {4},
	pages = {1934--1977},
	title = {Convergence and quasi-optimal complexity of adaptive finite element computations for multiple eigenvalues},
	url = {https://doi.org/10.1093/imanum/dru059},
	volume = {35},
	year = {2015}}

@article{MR3347459,
	author = {Gallistl, Dietmar},
	doi = {10.1007/s00211-014-0671-8},
	fjournal = {Numerische Mathematik},
	issn = {0029-599X},
	journal = {Numer. Math.},
	mrclass = {65N25 (65N12 65N30)},
	mrnumber = {3347459},
	mrreviewer = {Steve Wright},
	number = {3},
	pages = {467--496},
	title = {An optimal adaptive {FEM} for eigenvalue clusters},
	url = {https://doi.org/10.1007/s00211-014-0671-8},
	volume = {130},
	year = {2015}}

@article{MR4136540,
	author = {Canc\`es, Eric and Dusson, Genevi\`eve and Maday, Yvon and Stamm, Benjamin and Vohral\'{\i}k, Martin},
	doi = {10.1090/mcom/3549},
	fjournal = {Mathematics of Computation},
	issn = {0025-5718},
	journal = {Math. Comp.},
	mrclass = {65N25 (35J10 35P15 65N15 65N30)},
	mrnumber = {4136540},
	mrreviewer = {J. Manimaran},
	number = {326},
	pages = {2563--2611},
	title = {Guaranteed a posteriori bounds for eigenvalues and eigenvectors: multiplicities and clusters},
	url = {https://doi.org/10.1090/mcom/3549},
	volume = {89},
	year = {2020}}

@article{MR2928972,
	author = {Chen, Huangxin and Xu, Xuejun and Zheng, Weiying},
	doi = {10.4208/jcm.1109-m3401},
	fjournal = {Journal of Computational Mathematics},
	issn = {0254-9409},
	journal = {J. Comput. Math.},
	mrclass = {65N30 (65F10 65N22 65N55)},
	mrnumber = {2928972},
	number = {3},
	pages = {223--248},
	title = {Local multilevel methods for second-order elliptic problems with highly discontinuous coefficients},
	url = {https://doi.org/10.4208/jcm.1109-m3401},
	volume = {30},
	year = {2012}}

@article{MR745088,
	author = {Babu\v{s}ka, I. and Vogelius, M.},
	doi = {10.1007/BF01389757},
	fjournal = {Numerische Mathematik},
	issn = {0029-599X},
	journal = {Numer. Math.},
	mrclass = {65L60 (65L10)},
	mrnumber = {745088},
	mrreviewer = {W. C. Rheinboldt},
	number = {1},
	pages = {75--102},
	title = {Feedback and adaptive finite element solution of one-dimensional boundary value problems},
	url = {https://doi.org/10.1007/BF01389757},
	volume = {44},
	year = {1984}}

@article{MR2050077,
	author = {Binev, Peter and Dahmen, Wolfgang and DeVore, Ron},
	doi = {10.1007/s00211-003-0492-7},
	fjournal = {Numerische Mathematik},
	issn = {0029-599X,0945-3245},
	journal = {Numer. Math.},
	mrclass = {65N50 (65N12 65N30 65Y20 68W25 68W40)},
	mrnumber = {2050077},
	mrreviewer = {Thomas\ Apel},
	number = {2},
	pages = {219--268},
	title = {Adaptive finite element methods with convergence rates},
	url = {https://doi.org/10.1007/s00211-003-0492-7},
	volume = {97},
	year = {2004}}

@article{MR3499456,
	author = {Zhao, Tao and Hwang, Feng-Nan and Cai, Xiao-Chuan},
	doi = {10.1016/j.cpc.2016.03.009},
	fjournal = {Computer Physics Communications},
	issn = {0010-4655,1879-2944},
	journal = {Comput. Phys. Commun.},
	mrclass = {65F15 (81V65)},
	mrnumber = {3499456},
	pages = {74--81},
	title = {Parallel two-level domain decomposition based {J}acobi-{D}avidson algorithms for pyramidal quantum dot simulation},
	url = {https://doi.org/10.1016/j.cpc.2016.03.009},
	volume = {204},
	year = {2016}}

@article {MR2652780,
    AUTHOR = {Boffi, Daniele},
     TITLE = {Finite element approximation of eigenvalue problems},
   JOURNAL = {Acta Numer.},
  FJOURNAL = {Acta Numerica},
    VOLUME = {19},
      YEAR = {2010},
     PAGES = {1--120},
      ISSN = {0962-4929},
   MRCLASS = {65N30 (65N25)},
  MRNUMBER = {2652780},
MRREVIEWER = {Srinivasan Kesavan},
       DOI = {10.1017/S0962492910000012},
       URL = {https://doi.org/10.1017/S0962492910000012},
}

@article {MR1778354,
    AUTHOR = {Sleijpen, Gerard L. G. and Van der Vorst, Henk A.},
     TITLE = {A {J}acobi-{D}avidson iteration method for linear eigenvalue
              problems},
   JOURNAL = {SIAM Rev.},
  FJOURNAL = {SIAM Review},
    VOLUME = {42},
      YEAR = {2000},
    NUMBER = {2},
     PAGES = {267--293},
      ISSN = {0036-1445},
   MRCLASS = {65F15},
  MRNUMBER = {1778354},
       DOI = {10.1137/S0036144599363084},
       URL = {https://doi.org/10.1137/S0036144599363084},
}

@article {MR2970733,
    AUTHOR = {Carstensen, Carsten and Gedicke, Joscha},
     TITLE = {An adaptive finite element eigenvalue solver of asymptotic
              quasi-optimal computational complexity},
   JOURNAL = {SIAM J. Numer. Anal.},
  FJOURNAL = {SIAM Journal on Numerical Analysis},
    VOLUME = {50},
      YEAR = {2012},
    NUMBER = {3},
     PAGES = {1029--1057},
      ISSN = {0036-1429},
   MRCLASS = {65N25 (65N15 65N30)},
  MRNUMBER = {2970733},
MRREVIEWER = {V. L. Makarov},
       DOI = {10.1137/090769430},
       URL = {https://doi.org/10.1137/090769430},
}
\end{document}